\documentclass[a4paper,12pt]{article}

\usepackage{amsfonts,amssymb,amsbsy,amsmath,amsthm,mathrsfs,enumerate,verbatim}

  \usepackage{paralist}
  \usepackage{graphics} 
  \usepackage{epsfig} 
\usepackage{graphicx}  \usepackage{epstopdf}
 \usepackage[colorlinks=true]{hyperref}

\hypersetup{urlcolor=blue, citecolor=red}

\usepackage{amsfonts}
\usepackage{amsmath}
\usepackage{amssymb}
\usepackage{graphicx}
\usepackage{epsf}
\usepackage{epsfig}

\usepackage{color}
\usepackage{afterpage}
\usepackage{verbatim}

\newcommand{\norm}[1]{\left\Vert#1\right\Vert}
\newcommand{\abs}[1]{\left\vert#1\right\vert}

\usepackage{times}

\newtheorem{theo}{Theorem}

\newtheorem{theorem}{Theorem}[section]
\newtheorem{proposition}[theorem]{Proposition}
\newtheorem{lemma}[theorem]{Lemma}
\newtheorem{follow}[theorem]{Corollary}

\theoremstyle{definition}

\newcommand{\bel}{\begin{equation} \label}
\newcommand{\ee}{\end{equation}}

\newcommand{\pd}{\partial}

\newcommand{\R}{{\mathbb R}}
\newcommand{\N}{{\mathbb N}}

\newcommand{\bx}{{\bf x}}

\newcommand{\txi}{\tilde{\xi}}

\newcommand{\eps}{{\epsilon}}
\def\beq{\begin{equation}}
\def\eeq{\end{equation}}
\newcommand{\bea}{\begin{eqnarray}}
\newcommand{\eea}{\end{eqnarray}}
\newcommand{\beas}{\begin{eqnarray*}}
\newcommand{\eeas}{\end{eqnarray*}}

{

 \usepackage{graphicx,color}
 \definecolor{mygreen}{cmyk}{1,0,1,0.1}

\usepackage{epsf}
\usepackage{epstopdf}

\usepackage{pdfsync}

\begin{document}

\title{Uniqueness and stability of time and space-dependent conductivity in a
  hyperbolic cylindrical domain}

\author{L. Beilina  \thanks{
Department of Mathematical Sciences, Chalmers University of Technology and
Gothenburg University, SE-42196 Gothenburg, Sweden, e-mail: \texttt{\
larisa@chalmers.se}}
\and
M. Cristofol \thanks{Institut de Math\'{e}matiques de Marseille, CNRS, UMR 7373, \'Ecole Centrale, Aix-Marseille Universit\'e, 13453 Marseille, France}
\and
S. Li \thanks{Key Laboratory of Wu Wen-Tsun Mathematics, Chinese Academy of Sciences, School of Mathematical Sciences, University of Science and Technology of China, 96 Jinzhai Road, Hefei, Anhui Province, 230026,  China}}

\date{}

\maketitle

\begin{abstract}

  This paper is devoted to the reconstruction of the time and
  space-dependent coefficient in an infinite cylindrical hyperbolic
  domain. Using a local Carleman estimate we prove the uniqueness and
  a H\"older stability in the determining of the conductivity by a
  single measurement on the lateral boundary. Our numerical examples
  show good reconstruction of the location and contrast of the
  conductivity function in three dimensions.

\end{abstract}

\maketitle


\section{Introduction}

The present result is based on a recent work \cite{CLS15} dealing with
the inverse problem of determining the time-independent isotropic
conductivity coefficient $c : \Omega \to \R$ appearing in the
hyperbolic partial differential equation $(\partial_t^2 -\nabla \cdot
c \nabla)u=0$, where $\Omega := \omega \times \R$ is an infinite
cylindrical domain whose cross section $\omega$ is a bounded open
subset of $\R^{n-1}$, $n \geq 2$. Our goal is to extend this
reconstruction result, based on a finite number of observations, to a
more general class of conductivities: time and space-dependent
conductivities $\tilde{c}(x,t)$. The  reconstruction of
time and space-dependent coefficients in an infinite domain with a
finite number of observations is very challenging. The approach
developed here consists to retrieve any arbitrary bounded subpart of
the non-compactly supported conductivity $\tilde{c}$ from one data
taken on a compact subset of the lateral boundary $\Gamma=\partial
\Omega = \partial \omega \times (-\infty,\infty)$. Furthermore, a
stability inequality is established which links the distance between
two sets of coefficients $\tilde{c}_1(x,t)$ and $\tilde{c}_2(x,t)$
with the distance of lateral boundary observation of the Neumann
derivative of the solutions $u_1$ and $u_2$. First, this stability
inequality implies the uniqueness of the coefficient $\tilde{c}$ and
second, we can use it to perform numerical reconstruction using noisy
observations since real observations are generally noisy. Since a lot
of physical background involve time and space-dependent
conductivities, we extend the results in \cite{CLS15} by considering
the following initial boundary value problem \bel{S1}
\left\{ \begin{array}{ll} \pd_t^2 u - \nabla \cdot \tilde{c} \nabla u
  = 0 & \mbox{in}\ Q := \Omega \times (0,T), \\ u(\cdot,0) =
  \theta_0,\ \pd_t u(\cdot,0) = \theta_1 & \mbox{in}\ \Omega, \\ u = 0
  & \mbox{on}\ \Sigma:= \Gamma \times (0,T),
\end{array} \right.
\ee with initial conditions $(\theta_0,\theta_1)$, where $\tilde{c}$
is the unknown conductivity coefficient we aim to retrieve, and we
assume that $\tilde{c}$ is time and space depending in the following
form : \bel{ct} \tilde{c}(x,t) = c_0(x,t) + c(x), \ee where $c_0(x,t)$
is assumed to be known. That means that we consider the case of the
perturbation of a general time and space-dependent conductivity by a
space-dependent one. Such model is not a direct application of known
results and involves several technical difficulties connected to the
time dependence which will be detailed later. For a
similar general non-stationary media, we can refer to \cite{LY13}
where the authors study an inverse problem for Maxwell's
equations.

Several stability results in the inverse problem of
determining one or several unknown coefficients of a hyperbolic  
equation from a finite number of measurements of the solution are
available in the mathematics literature, see
\cite{Be5,Be6,BJY08,BY06,BeYa08,IY2,IY03,K01,K02,KY06, Y99} for
example, and their derivation relies on a Carleman inequality
specifically designed for hyperbolic systems. All these works concern
space-dependent coefficients.

On the other hand, no of these works
are associated to numerical simulations whereas the existence of a
stability inequality allows to improve the resolution of the
minimization problem by choosing more precisely the functional to
minimize.  Furthermore, the case of the reconstruction of the
conductivity coefficient in the divergence form for the hyperbolic
operator induces some numerical difficulties, see
\cite{BAbsorb,BNAbsorb, BH, Chow} for details.  A lot of papers are
dealing with an optimization approach without any theoretical study
and in most of the cases, the uniqueness of the associated inverse
problem is not proved.
On the other hand, we develop in this paper numerical simulations for
a problem of the reconstruction of the conductivity coefficient in the
form \eqref{ct} based on partial boundary observations similar to
whose used in our theoretical result.

In previous works \cite{BAbsorb,BNAbsorb} were presented
numerical studies of the reconstruction of the space-dependent
conductivity function in a hyperbolic equation using backscattered
data in three dimensions.  Also in \cite{Chow} were presented
numerical simulations of reconstruction of the only space-dependent
conductivity function in two-dimensions. In \cite{Chow} a
layer-stripping procedure was used instead of the Lagrangian approach
of \cite{BAbsorb,BNAbsorb}.  However, the time-dependent function as a
part of the conductivity function was not considered in the above
cited works.

In our numerical examples  of this work we tested the reconstruction of a
conductivity function that represents a sum of two space-dependent
gaussians and one time-dependent function. Since by our assumption the
time-dependent function is known inside the domain, then the goal of
our numerical experiments is to reconstruct only the space-dependent
part of the conductivity function. To do that we used Lagrangian
approach together with the domain decomposition finite element/finite
difference method of \cite{BAbsorb}. One of the important points of
this work is that in our numerical simulations we applied one non-zero
initial condition in the model problem what corresponds well to the
uniqueness and stability results of this paper.

Our three-dimensional numerical simulations show that we can
accurately reconstruct location and large contrast of the
space-dependent function which is a part of the known time-dependent
function. However, the location of this function in the third, $x_3$
direction, should be still improved.  Similarly with \cite{B, BTKB, BH} we
are going to apply an adaptive finite element method to improve
reconstruction of shape of the space-dependent function obtained in
this work.

The outline of the work is the following: in section \ref{intro} we
prove the uniqueness and stability result for the system \eqref{S1},
 in section \ref{sec:Numer-Simul} we present our numerical
simulations, and in section \ref{sec:concl}   we summarize results of our work.

\section{Mathematical background and main theoretical result}
\label{intro}
\setcounter{equation}{0}

\subsection{Notations and hypothesis}
\label{sec-intro}

Throughout this article, we keep the following notations: $x=(x',x_n) \in \Omega$ for every $x':=(x_1,\ldots,x_{n-1}) \in \omega$ and $x_n \in \R$.
Further, we denote by $| y |:=\left( \sum_{i=1}^m y_j^2 \right)^{1 \slash 2}$ the Euclidean norm of $y=(y_1,\ldots,y_m) \in \R^m$, $m \in \N^*$, and we write $\mathbb{S}^{n-1}:=\left\{ x'=(x_1,\ldots,x_{n-1}) \in \R^{n-1},\ |x'|=1 \right\}$.
We write $\pd_j$ for $\pd \slash \pd x_j$, $j \in \N_{n+1}^*:=\{ m \in \N^*,\ m \leq n+1 \}$.
For convenience the time variable $t$ is sometimes denoted by $x_{n+1}$ so that $\pd_t=\pd \slash \pd t=\pd_{n+1}$.
We set $\nabla:=(\pd_1,\ldots,\pd_n)^T$, $\nabla_{x'}:=(\pd_1,\ldots,\pd_{n-1})^T$ and
$\nabla_{x,t}=(\pd_1,\ldots,\pd_{n},\pd_t)^T$. 

For any open subset $D$ of $\R^m$, $m \in \N^*$, we note $H^p(D)$ the $p$-th order Sobolev space on $D$ for every $p \in \N$, where
$H^0(D)$ stands for $L^2(D)$. We write $\| \cdot \|_{p,D}$ for the usual norm in $H^p(D)$ and we note
$H_0^1(D)$ the closure of $C_0^{\infty}(D)$ in the topology of $H^1(D)$.

Finally, for $d>0$ we put $\Omega_d := \omega \times (-d,d)$, $Q_d:=\Omega_d \times(0, T)$, $\Gamma_{d}:= \partial \omega \times (-d, d)$ and $\Sigma_{d} := \partial \omega \times(-d, d) \times (0, T)$.

We are interested by the initial boundary value problem  (\ref{S1}). 
We shall suppose that $\tilde{c} $ fulfills the ellipticity condition
\bel{c0a}
\tilde{c} \geq c_m\ \mbox{in}\ Q,
\ee
for some positive constant $c_m$.
In order to solve the inverse problem associated with \eqref{S1} we seek solutions belonging to $\cap_{k=3}^4 C^k([0,T];H^{5-k}(\Omega))$.
Following the strategy used  in \cite{CLS15} based on the reference \cite[Sect. 3, Theorem 8.2]{LM1}, it is sufficient  to assume  that the coefficient  $\tilde{c} $ (resp. $c$)   be  in  $C^{\infty}(Q;\R)$ (resp.  be in $C^{\infty}(\Omega;\R)$)  and $\partial \omega \in C^{\infty}$,  to get the required regularity  for the solution $u$ of the system \eqref{S1}. We note $c_M$ a positive constant fulfilling
\bel{c0b}
 \| \tilde{c}  \|_{W^{4,\infty}(Q)} \leq c_M.
\ee
Since our strategy is based on a Carleman estimate for the hyperbolic system \eqref{S1}, it is also required that the condition
\bel{c1}
a' \cdot \nabla_{x'} \tilde{c} \geq \mathfrak{a}_0\ \mbox{in}\ Q,
\ee
holds for some $a'=(a_1,\ldots,a_{n-1}) \in \mathbb{S}^{n-1}$ and $\mathfrak{a}_0>0$.
Hence, given
$\omega^{\#}$ an open subset of 
$\R^{n-1}$ such that $\partial \omega \subset \omega^{\#} $, we put 
$\mathcal{O}_*=\omega^{\#}\times\R$,
and for $c_* \in C^{\infty}((\mathcal{O}_* \cap \Omega) \times(0,T);\R)$ satisfying
\bel{c0a*}
c_* \geq c_m\ \mbox{and}\ a' \cdot \nabla_{x'} c_* \geq \mathfrak{a}_0\ \mbox{in}\ (\mathcal{O}_* \cap \Omega) \times(0,T),
\ee
we introduce the set $\Lambda_{\mathcal{O}_*}=\Lambda_{\mathcal{O}_*}(a',\mathfrak{a_0},c_*,c_m,c_M)$ of admissible conductivity coefficients as
\bel{1.7}
\Lambda_{\mathcal{O}_*}:=\{ \tilde{c} \in W^{4,\infty}(Q;\R)\ \mbox{obeying}\  \eqref{ct}-\eqref{c1};\ \tilde{c}=c_*\ \mbox{in}\ (\mathcal{O}_* \cap \Omega) \times (0,T) \}.
\ee

Furthermore, it is required  that $\theta_0$ be in $W^{3,\infty}(\Omega) \cap H^5(\Omega)$ and satisfy
\bel{S10}
- a' \cdot \nabla_{x'} \theta_0 \geq \eta_0 e^{-(1+x_n^2)},x=(x',x_n) \in \ \mbox{in}\ 
\omega_*\times \R,
\ee
for some $\eta_0>0$ and some open subset $\omega_*$ in $\R^{n-1}$, with $C^2$ boundary, satisfying
\bel{ic0}
\overline{\omega \setminus (\omega^{\#}\cap\omega)} \subset \omega_* \mbox{ and } \overline{\omega_*} \subset \omega,
\ee
and that exists $M_0>0$ such that
\bel{S100}
\| \theta_0 \|_{W^{3,\infty}(\Omega)} +  \| \theta_0 \|_{H^5(\Omega)} \leq M_0.
\ee

\subsection{Main result}

 The following result claims H\"older stability in the inverse problem of determining $c$ in $\Omega_\ell$, where $\ell>0$ is arbitrary, from the knowledge of one boundary measurement  of the solution to \eqref{S1}, performed on $\Sigma_L$ for $L > \ell$ sufficiently large. The corresponding observation is viewed as a vector of the Hilbert space
$$\mathscr{H}( \Sigma_{L}) := H^3(0,T;L^2(\Gamma_{L})), $$
endowed with the norm,

$$ \norm{v}^2_{\mathscr{H}( \Sigma_{L})} :=\norm{v}^2_{H^3(0,T;L^2(\Gamma_{L}))},\ v\in
\mathscr{H}( \Sigma_{L}). $$

\begin{theorem}
\label{T.1}
Assume that $\pd \omega$ is $C^5$ and let $\mathcal{O}_*$ be a neighborhood of $\Gamma$ in $\R^{n-1}$. Assume that $c_0 \in W^{4,\infty}(Q;\R)$ and $\partial_t c_0(\cdot,0) = \partial_t^3 c_0(\cdot,0) = 0$ in $\Omega$. For $a'=(a_1,\ldots,a_{n-1}) \in \mathbb{S}^{n-1}$, $\mathfrak{a}_0>0$, $c_m \in (0,1)$, $c_M>c_m$ and $c_* \in W^{4,\infty}((\mathcal{O}_* \cap \Omega) \times (0,T);\R)$ fulfilling \eqref{c0a*}, pick $c_j$, $j=1,2$, in $\Lambda_{\mathcal{O}_*}(a',\mathfrak{a_0},c_*,c_m,c_M)$, defined by \eqref{1.7}. Further, given $M_0>0$, $\eta_0>0$ and an open subset $\omega_* \subset \R^{n-1}$ obeying \eqref{ic0}, let
$\theta_0$ fulfill \eqref{S10}-\eqref{S100}, and $\theta_1 = 0$.

Then for any $\ell>0$ we may find $L >\ell$ and $T>0$,
such that the $\bigcap_{k=0}^5 C^k([0,T],H^{5-k}(\Omega))$-solution $u_j$, $j=1,2$, to \eqref{S1} associated with $(\theta_0,\theta_1)$, where $\tilde{c}_j$ is substituted for $\tilde{c}$, satisfies
$$
\norm{\tilde{c}_1-\tilde{c}_2}_{H^1(\Omega_\ell)} \leq C  \norm{ \frac{\partial u_1}{\partial \nu}-\frac{\partial u_2}{\partial \nu}}_{\mathscr{H}( \Sigma_{L})}^\kappa.
$$
Here $C>0$ and $\kappa \in (0,1)$ are two constants depending only on $\omega$, $\ell$, $M_0$, $\eta_0$, $a'$, $\mathfrak{a}_0$, $c_\star$, $c_m$ and $c_M$.
\end{theorem}

The main difficulties associated to the time dependence of the conductivity coefficient $\tilde{c}$ appear on one hand in the proof of the Carleman estimate for second order hyperbolic operators. On the other hand, the Bukhgeim-Klibanov method uses intensively time differentiation. In the case of time and space-dependent coefficients we have to manage a lot of additive terms with respect to the case of only space-dependent conductivity. Nevertheless, the result obtained in  Theorem \ref{T.1} is similar to the main theorem established in \cite{CLS15}.

\subsection{A Carleman estimate for second order hyperbolic operators with time dependent coefficient in cylindrical domains}

In this section we establish a global Carleman estimate for the system \eqref{S1} and in view of the inverse problem,  we start by time-symmetrizing the solution $u$ of \eqref{S1}. Namely, we put
\bel{t-sym}
u(x,t) := u(x,-t),\ x \in \Omega,\ t \in (-T,0).
\ee
and
\bel{c-sym}
c_0(x,t) := c_0(x,-t),\ x \in \Omega,\ t \in (-T,0).
\ee

We consider the operator
\bel{c1b}
A:=A(x,t,\pd)=\partial_t^2 - \nabla \cdot \tilde{c} \nabla + R,
\ee
where $R$ is a first-order partial differential operator with $L^\infty(Q)$ coefficients.  For simplicity, we put $Q:=\Omega\times(-T,T)$, $\Sigma:=\Gamma\times(-T, T)$, and $\Sigma_L:=\partial\omega\times(-L,L)\times(-T,T)$ for any $L>0$, in the remain part of this section.

 We define for every $\delta>0$, $\gamma>0$ , and $a'\in \mathbb{S}^{n-1}$ fulfilling \eqref{c1}, the following weight functions:
\bel{c3}
\psi(x,t)=\psi_\delta(x,t):=| x'- \delta a' |^2-x_n^2-t^2\ \mbox{and}\
\varphi(x,t)=\varphi_{\delta,\gamma}(x,t):={\rm e}^{\gamma \psi(x,t)},\ (x,t) \in Q.
\ee

\begin{proposition}
\label{pr-ic}
Let $A$ be defined by \eqref{c1b}, where $\tilde{c}$ verifies \eqref{c0a}--\eqref{c1}, and let $\ell$ be positive. Then there exist $\delta_0>0$ and $\gamma_0>0$ such that for all
$\delta \geq \delta_0$ and $\gamma \geq \gamma_0$, we may find $L>\ell$, $T>0$ and $s_0>0$ for which the estimate
\bel{c16}
s \sum_{j=0,1} s^{2(1-j)} \| {\rm e}^{s \varphi} \nabla_{x,t}^j v \|_{0,Q_L}^2
\leq C \left( \| {\rm e}^{s\varphi} A v \|_{0,Q_L}^2 + s \sum_{j=0,1} s^{2(1-j)} \| {\rm e}^{s\varphi} \nabla_{x,t}^j v \|_{0,\pd Q_L}^2 \right),
\ee
holds for any $s \geq s_0$ and $v \in H^2\left(Q_L\right)$.
Here $C$ is a positive constant depending only on $\omega$, $a'$, $\mathfrak{a}_0$, $\delta_0$, $\gamma_0$, $s_0$, $c_m$ and $c_M$.

Moreover there exists a constant $d_\ell>0$, depending only on $\omega$, $\ell$, $\delta_0$ and $\gamma_0$, such that the weight function $\varphi$ defined by \eqref{c3} satisfies
\bel{d5}
\varphi(x', x_n, 0) \geq d_\ell,\ (x', x_n) \in \overline{\omega} \times [-\ell, \ell],
\ee
and we may find $\eps \in (0, (L-\ell) \slash 2)$ and $\zeta >0$ so small that we have:
\bea
\max_{x \in\overline{\omega}\times [-L, L]}
\varphi(x', x_n, t) & \leq & \tilde{d}_\ell:=d_\ell {\rm e}^{-\gamma  \zeta^2},\ |t|\in [T-2\eps, T], \label{d6} \\
\max_{(x', t) \in \overline{\omega} \times [-T, T]} \varphi(x', x_n, t) & \leq & \tilde{d}_\ell,\ |x_n|\in [L-2\eps, L]. \label{d7}
\eea
\end{proposition}

\begin{proof}
We mimic  a part of the proof from our previous paper \cite{CLS15}, first adapting the definition of $\delta_0$ in \eqref{d2} because the time dependence of the conductivity introduces several difficulties. We stress out the main change with respect to the time independent version.

We define first  $\delta_0$, $L$ and $T$ and following the notations introduced in \cite{CLS15} we introduce 
\bel{d1}
g_\ell(\delta) = \left( \sup_{x' \in \omega} |x' - \delta a' |^2 - \inf_{x' \in \omega} |x' - \delta a' |^2 + \ell^2 \right)^{1 \slash 2}
\ee
then exists $\delta_0>0$ so large that

\bel{d2}
\delta \mathfrak{a}_0 >  \left( \left( 1 + \frac{2\sqrt{n} }{{c_m}^{1 \slash 2}} \right) g_\ell(\delta) + \sqrt{n-1} | \omega |+ 2 \right) c_M + 2+ \frac{g_\ell(\delta)}{c_m} c_M,\ \delta \geq \delta_0,
\ee
where $\mathfrak{a}_0$, $a'$ are introduced in \eqref{c1}.\\

Further, since $\omega$ is bounded and $a' \neq 0_{\R^{n-1}}$ by \eqref{c1}, we may as well assume upon possibly enlarging $\delta_0$, that we have in addition
$c_m^{1 \slash 2} \inf_{x' \in \overline{\omega}} | x' - \delta a' | > g_\ell(\delta)$ for all $\delta \geq \delta_0$.
This and \eqref{d2} yield that there exists $\vartheta>0$ so small that the two following inequalities

\bel{c2}
\delta \mathfrak{a}_0 - \left( L+ \sqrt{n-1}| \omega | + 2 \left( 1 + \frac{\sqrt{n} T}{{c_m}^{1 \slash 2}} \right) \right) c_M - 2 -  \frac{T}{c_m}c_M>0,
\ee

and
\bel{c4}
c_m^{1 \slash 2} \inf_{x' \in \overline{\omega}} | x' - \delta a' | > T,
\ee
hold simultaneously for every $L$ and $T$ in $(g_\ell(\delta), g_\ell(\delta)+\vartheta)$, uniformly in $\delta \geq \delta_0$.\\
It can be underlined that the time dependence of $\tilde{c}$ implies to reformulate the equations \eqref{d2} and \eqref{c2} with respect to the similar one in \cite{CLS15}.

Now, we come back to the proof of \eqref{c16}. 
Our approach is based  on Isakov \cite[Theorem 3.2.1']{I}. We put $\bx:=(x,t)$ for $(x,t) \in Q_L$ and $\nabla_{\bx}=(\pd_1,\ldots,\pd_{n},\pd_{n+1})^T$. We also write $\xi'=(\xi_1,\ldots,\xi_{n-1})^T \in \R^{n-1}$, $\xi=(\xi_1,\ldots,\xi_n)^T \in \R^n$ and $\txi=(\xi_1,\ldots,\xi_n,\xi_{n+1})^T \in \R^{n+1}$. We call $A_2$ the principal part of the operator $A$, that is
$A_2=A_2(\bx,\pd)=\pd_t^2- \tilde{c}(x,t) \Delta$, and denote its symbol by $A_2(\bx,\txi)=\tilde{c}(x,t) | \xi |^2 - \xi_{n+1}^2$, where $| \xi |= \left( \sum_{j=1}^n \xi_j^2 \right)^{1 \slash 2}$.
Since $A_2(\bx,\nabla_{\bx} \psi(\bx)) = 4 \left( \tilde{c}(x,t) ( | x' - \delta a' |^2 + x_n^2 ) - x_{n+1}^2 \right)$ for every $\bx \in \overline{Q}_L$, we have
\bel{c10}
A_2(\bx,\nabla_{\bx} \psi(\bx))>0,\ \bx \in \overline{Q}_L,
\ee
by \textcolor{blue}{\eqref{c0a}} and \eqref{c4}.
For all $\bx \in \overline{Q}_L$ and $\txi \in \R^{n+1}$, put
\bel{c11}
J(\bx,\txi) = J =\sum_{j,k=1}^{n+1} \frac{\pd A_2}{\pd \xi_j} \frac{\pd A_2}{\pd \xi_k} \pd_j \pd_k \psi +
\sum_{j,k=1}^{n+1} \left( \left( \pd_k \frac{\pd A_2}{\pd \xi_j} \right) \frac{\pd A_2}{\pd \xi_k} - (\pd_k A_2) \frac{\pd^2 A_2}{\pd \xi_j \pd \xi_k}
\right) \pd_j \psi,
\ee
where we write $\partial_j$, $j \in \N_{n+1}^*$, instead of $\pd \slash \pd x_j$, and $x_{n+1}$ stands for $t$.
We assume that
\bel{c12}
A_2(\bx,\txi)=\tilde{c}(x,t) | \xi |^2 - \xi_{n+1}^2 = 0,\ x \in \overline{\Omega},\ t \in (0,T),\ \txi \in \R^{n+1} \backslash \{ 0 \},
\ee
and that 
\bel{c13}
\nabla_{\txi} A_2(x,\txi) \cdot \nabla_{\bx} \psi(\bx) = 4 \left[ \tilde{c}(x,t) ( \xi' \cdot (x' - \delta a') - \xi_n x_n ) + \xi_{n+1} x_{n+1} \right]=0,\
\ee
for $\bx \in \overline{Q}_L,\ \txi \in \R^{n+1} \backslash \{ 0 \}$
 and we shall prove that $J(\bx,\txi)>0$ for any $(\bx,\txi) \in \overline{Q}_L \times
\left\{\R^{n+1}\setminus\{0\}\right\}$.\\
 To this end we notice that the first sum in the right hand side of \eqref{c11} reads
$$\langle \mbox{Hess}(\psi) \nabla_{\txi} A_2 , \nabla_{\txi} A_2 \rangle = 8 \left( \tilde{c}^2 ( | \xi' |^2 - \xi_n^2 ) - \xi_{n+1}^2 \right),$$
and that 

\beas
& \sum_{j,k=1}^{n+1} \left( \left( \pd_k \frac{\pd A_2}{\pd \xi_j} \right) \frac{\pd A_2}{\pd \xi_k} - (\pd_k A_2) \frac{\pd^2 A_2}{\pd \xi_j \pd \xi_k} \right) \pd_j \psi =
   2 \tilde{c} \left( 2 (\nabla \tilde{c} \cdot \xi) (\nabla \psi \cdot \xi)  -  (\nabla \tilde{c} \cdot \nabla \psi) | \xi |^2 \right) \\
&- 4 (\partial_{n+1} \tilde{c}) \xi_{n+1} (\nabla\psi \cdot \xi) + 2 (\partial_{n+1} \tilde{c}) (\partial_{n+1} \psi) |\xi |^2,
\eeas
since from the time dependence of $\tilde{c}$, 
$$\left(\pd_k \frac{\pd A_2}{\pd \xi_j} \right) \frac{\pd A_2}{\pd \xi_k} - (\pd_k A_2) \frac{\pd^2 A_2}{\pd \xi_j \pd \xi_k}\neq 0$$
 if either $j$ or $k$ is equivalent to $n+1$.

Therefore we have
\beas
J  &= & 4 \left[ 2 \tilde{c}^2 (|\xi'|^2 - \xi_{n}^2) - \left( 2 + (x'-\delta a') \cdot \nabla_{x'} \tilde{c} - x_n \pd_n \tilde{c} \right) \xi_{n+1}^2 - 2 x_{n+1} \xi_{n+1} (\nabla \tilde{c} \cdot \xi) \right]\\
 &&   + 8 \frac{\partial_{n+1}\tilde{c}}{\tilde{c}} x_{n+1}\xi_{n+1}^2 -4t (\partial_{n+1} \tilde{c})  |\xi |^2
\eeas
from \eqref{c12}-\eqref{c13}. Further, in view of \eqref{c12} we have
 $$\tilde{c}^2 ( | \xi'|^2 - \xi_n^2 ) \geq - \tilde{c}^2 | \xi |^2 \geq -\tilde{c} \ \xi_{n+1}^2 $$
  and
$$| \nabla \tilde{c} \cdot \xi | \leq | \nabla \tilde{c}| | \xi |  \leq ( | \nabla \tilde{c} | \slash \tilde{c}^{1 \slash 2} )| \xi_{n+1} |,$$
and the additive term with respect to the paper [CLS] is underestimated by
  $$-4\frac{T}{\tilde{c}}|\partial_{n+1} \tilde{c} | \xi_{n+1}^2,$$
   then,
\bel{c15}
J \geq 4 \left[ \delta  a' \cdot \nabla_{x'} \tilde{c} -  \left( x' \cdot \nabla_{x'} \tilde{c} -x_n\partial_n \tilde{c}+ 2 \tilde{c} + 2 T \frac{| \nabla \tilde{c} |}{\tilde{c}^{1 \slash 2}} + 2 +  \frac{T}{\tilde{c}}|\partial_{n+1} \tilde{c}| \right) \right] \xi_{n+1}^2.
\ee

Here we used the fact that $x_{n+1}=t \in [0,T]$. Due to \eqref{c0a}-\eqref{c1}, the right hand side of \eqref{c15} is lower bounded, up to the multiplicative constant $4 \xi_{n+1}^2$, by the left hand side of \eqref{c2}. Since $\xi_{n+1}$ is non zero by \eqref{c0b} and \eqref{c12}, then we obtain $J(\bx,\txi)>0$ for all $(\bx,\txi) \in \overline{Q}_L \times \left\{\R^{n+1}\setminus\{0\}\right\}$. With reference to
\eqref{c10}, we may apply \cite[Theorem 3.2.1']{I}, getting two constants $s_0=s_0(\gamma)>0$ and $C>0$ such that \eqref{c16}
holds for any $s \geq s_0$ and $v \in H^2(Q_L)$.

The end of the proof is similar  to the third part in the proof of Proposition 3.1 of \cite{CLS15}.

\end{proof}
Now we can derive from Proposition \ref{pr-ic} a global Carleman estimate for the solution to the boundary value problem
\bel{c18}
\left\{  \begin{array}{ll} \pd_t^2 u - \nabla \cdot \tilde{c} \nabla u = f & \mbox{in}\ Q, \\
  u= 0 & \mbox{on}\ \Sigma,
\end{array} \right.
\ee
where $f \in L^2(Q)$. 
To this purpose we introduce a cut-off function $\chi \in C^2(\R;[0,1])$, such that
\bel{c21}
\chi(x_n):= \left\{ \begin{array}{cl} 1 & \mbox{if}\ |x_n| < L - 2 \eps, \\ 0 & \mbox{if}\ |x_n| \geq L - \eps, \end{array} \right.
\ee
where $\eps$ is the same as in Proposition \ref{pr-ic}, and we set
$$u_\chi(x,t):=\chi(x_n) u(x,t)\ \mbox{and}\ f_\chi(x,t):=\chi(x_n) f(x,t),\ (x,t) \in Q. $$

\begin{follow}
\label{cor-ec}

Let $f \in L^2(Q)$. Then, under the conditions of Proposition \ref{pr-ic},
there exist two constants $s_*>0$ and $C>0$, depending only on $\omega$, $\ell$, $M_0$, $\eta_0$, $a'$, $\mathfrak{a}_0$, $c_m$ and $c_M$, such that the estimate
$$
s \sum_{j=0,1} s^{2(1-j)} \| {\rm e}^{s \varphi} \nabla_{x,t}^j u \|_{0,Q_L}^2
\leq C \left(  \| {\rm e}^{s \varphi} f \|_{0,Q_L}^2 + s^3 {\rm e}^{2 s \tilde{d}_\ell}  \| u \|_{1,Q_L}^2  + s \sum_{j=0,1} s^{2(1-j)} \| {\rm e}^{s \varphi} \nabla_{x,t}^j u_\chi \|_{0,\Sigma_L}^2 \right),
$$
holds for any solution $u$ to \eqref{c18}, uniformly in $s \geq s_*$.
\end{follow}
For the proof see \cite{CLS15}.

\subsection{Inverse problem}
\label{sec-lip}

In this subsection we introduce the linearized inverse problem associated with \eqref{S1} and relate the first Sobolev norm of the conductivity to some suitable initial condition of this boundary problem.

Namely, given $\tilde{c}_i \in \Lambda_{\mathcal{O}_*}$ for $i=1,2$, we note $u_i$ the solution to \eqref{S1} where $\tilde{c}_i$ is substituted for $\tilde{c}$, suitably extended to $(-T,0)$ in accordance with \eqref{c-sym}. Thus, putting
\bel{s0}
c:=\tilde{c}_1-\tilde{c}_2\ \mbox{and}\ f_c:=\nabla \cdot ( c \nabla u_2 ),
\ee
it is clear from \eqref{S1} that the function $u:=u_1-u_2$ is solution to the linearized system
\bel{s1}
\left\{  \begin{array}{ll} \pd_t^2 u - \nabla \cdot ( \tilde{c}_1 \nabla u )  =  f_c & \mbox{in}\ Q, \\
  u  =  0 & \mbox{on}\ \Sigma, \\
u(\cdot,0) = \pd_t u(\cdot,0)  =  0 & \mbox{in}\ \Omega.
\end{array} \right.
\ee
Note that $c$ is time independent but $\tilde{c}_1$ is time dependent and this time dependence of a coefficient in the system \eqref{s1} implies to rewrite carefully the method used in \cite{CLS15}.

By differentiating $k$-times \eqref{s1} with respect to $t$, for $k \in \N^*$ fixed, we see that $u^{(k)}:=\pd_t^k u$ is solution to

\bel{s2}
\left\{  \begin{array}{l}
  \begin{split}
 \pd_t^2 u^{(k)} &- \nabla \cdot ( \tilde{c}_1 \nabla u^{(k)} ) +R_k(\pd_t  u^{(1)}, \pd_t  u^{(2)}, \cdots, \pd_t  u^{(k)},  \nabla u^{(0)},  \nabla u^{(1)}, \cdots,  \nabla u^{(k-1)})  \\
 &=  P_k(f_c^{(0)}, f_c^{(1)}, \cdots, f_c^{(k)}) \;\; \mbox{in}\ Q, \\
 u^{(k)}  &= 0  \;\; \mbox{on}\ \Sigma, \\
 \end{split}
\end{array} \right.
\ee
with $f_c^{(j)}:=\pd_t^j f_c=\nabla \cdot ( c \nabla u_2^{(j)} )$, $u_2^{(j)}=\partial_t^ju_2, j=0, 1, 2, ..., k, $ where  $R_k$ and $P_k$ stand for zero order operators with coefficients in $L^{\infty}(Q)$. 
The result (\ref{s2}) can be proved by the method of induction. 

In this part, a lot additive terms appear in the source term due to the time dependence of the conductivity $\tilde{c}_i$ and we have to manage them now very precisely.
We keep  the notations of Corollary \ref{cor-ec}. In particular, for any function $v$, we denote $\chi(x_n) \cdot v$ by $v_{\chi}$, where $\chi(x_n)$ is defined in \eqref{c21}.
Upon multiplying both sides of the identity \eqref{s2} by $\chi(x_n)$, we obtain that

\bel{s4}
\left\{  \begin{array}{l}
  \begin{split}
 \pd_t^2 u_{\chi}^{(k)} &- \nabla \cdot ( \tilde{c}_1 \nabla u_{\chi}^{(k)} ) +R_k(\pd_t  u_{\chi}^{(1)}, \pd_t  u_{\chi}^{(2)}, \cdots, \pd_t  u_{\chi}^{(k)},  \nabla u_{\chi}^{(0)},  \nabla u_{\chi}^{(1)}, \cdots,  \nabla u_{\chi}^{(k-1)})  \\
 &=  P_k(f_{c_{\chi}}^{(0)}, f_{c_{\chi}}^{(1)}, \cdots, f_{c_{\chi}}^{(k)}) - g_k \;\; \mbox{in}\ Q, \\
 u^{(k)}  &= 0  \;\; \mbox{on}\ \Sigma, \\
\end{split}
  \end{array} \right.
\ee
with
\bel{s4b}
f_{c_\chi}^{(j)} := \nabla \cdot (c_\chi \nabla u_2^{(j)}), \qquad j=0, 1, 2, ..., k,
\ee
and $g_k$ is supported in
$\tilde{Q}_{\epsilon}:= \{ x=(x',x_n, t),\ x' \in \omega, \ \abs{x_n} \in (L-2 \eps, L-\eps), \mbox{ and } t\in(-T,T)\}$.

Having said that we may now upper bound, up to suitable additive and multiplicative constants, the $e^{s \varphi(\cdot,0)}$-weighted first Sobolev norm of the conductivity $c_\chi$ in $\Omega_L$, by the corresponding norm of the initial condition $u_\chi^{(2)}(\cdot,0)$.

\begin{lemma}
\label{lm1}
Let $u$ be the solution to the linearized problem \eqref{s1} and let $\chi$ be defined by \eqref{c21}. Then there exist two constants $s_*>0$
and $C>0$, depending only on $\omega$, $\varepsilon$ and the constant $M_0$ defined by \eqref{S100}, such that the estimate
$$ \sum_{j=0,1} \| e^{s \varphi(\cdot,0)} \nabla^j c_{\chi} \|_{0,\Omega_L}^2 \leq
C s^{-1} \left( \sum_{j=0,1} \| e^{s \varphi(\cdot,0)} \nabla^j u_{\chi}^{(2)}(\cdot,0) \|_{0,\Omega_L}^2  + e^{2 s \tilde{d}_\ell} \right),
$$
holds for all $s \geq s_*$.
\end{lemma}
For the proof see \cite{CLS15}.

Now, we end the proof of theorem \ref{T.1}. 
We  first give an upper bound $u_\chi^{(2)}(\cdot,0)$ in the $e^{s \varphi(\cdot,0)}$-weighted $H^1(\Omega_L)$-norm topology, by the corresponding norms of $u_\chi^{(2)}$ and $u_\chi^{(3)}$ in $Q_L$.
\begin{lemma}
\label{lm2}
There exists a constant $s_*>0$ depending only on $T$  such that we have
$$ \| z(\cdot,0) \|_{0,\Omega_L}^2 \leq 2 \left( s  \| z \|_{0,Q_L}^2 + s^{-1} \| \pd_t z \|_{0,Q_L}^2 \right), $$
for all $s \geq s_*$ and $z \in H^1(-T,T;L^2(\Omega_L))$.
\end{lemma}
We apply Lemma \ref{lm2} with $z = e^{s \varphi} \pd_i^j u_\chi^{(2)}$ for $i \in \N_n^*$ and $j=0,1$, getting
$$
\| e^{s \varphi(\cdot,0)} \pd_i^j u_\chi^{(2)}(\cdot,0) \|_{0,\Omega_L}^2 \leq C \left( s  \| e^{s \varphi} \pd_i^j u_\chi^{(2)} \|_{0,Q_L}^2 + s^{-1} \| e^{s \varphi} \pd_i^j u_{\chi}^{(3)} \|_{0,Q_L}^2 \right),\ s \geq s_*.
$$
Summing up the above estimate over $i \in \N_n^*$ and $j=0,1$, we obtain for all $s \geq s_*$ that
\bel{s10}
\sum_{j=0,1} \| e^{s \varphi(\cdot,0)} \nabla^j u_\chi^{(2)}(\cdot,0) \|_{0,\Omega_L}^2
\leq C \sum_{j=0,1} \left( s \| e^{s \varphi} \nabla^j u_\chi^{(2)} \|_{0,Q_L}^2 + s^{-1} \| e^{s \varphi} \nabla^j u_\chi^{(3)} \|_{0,Q_L}^2 \right).
\ee

Then we majorize the right hand side of \eqref{s10} with
\bel{hs}
\mathfrak{h}_k(s) := \sum_{j=0,1} s^{2(1-j)} \| {\rm e}^{s \varphi} \nabla_{x,t}^j u_\chi^{(k)} \|_{0,\Sigma_L}^2,\ k=0,1,2,3.
\ee
Indeed, since $u_{\chi}^{(k)}$, for $k=0,1,2,3$,  is solution to \eqref{c18} with $\tilde{c}=\tilde{c}_1$ and $f=P_k(f_{c_\chi}^{(0)},  f_{c_\chi}^{(1)}, \cdots,  f_{c_\chi}^{(k)})- g_k - R_k(\partial_tu_\chi^{(1)}, \partial_tu_\chi^{(2)}, ..., \partial_tu_\chi^{(k)}, \nabla u_\chi^{(0)}, \nabla u_\chi^{(1)}, ..., \nabla u_\chi^{(k-1)})$, according to \eqref{s4}, then Corollary \ref{cor-ec} yields
\bea \nonumber
\begin{split}
&\sum_{k=0}^3 \left( s \sum_{j=0,1} s^{2(1-j)} \| {\rm e}^{s \varphi} \nabla_{x,t}^j u_\chi^{(k)} \|_{0,Q_L}^2\right) \nonumber\\
  &\leq C\sum_{k=0}^3 \left( \| {\rm e}^{s \varphi} f_{c_\chi}^{(k)} \|_{0,Q_L}^2 + \| {\rm e}^{s \varphi} g_k \|_{0,\tilde{Q}_{\eps}}^2+  s^3{\rm e}^{2 s \tilde{d}_\ell} \| u_\chi^{(k)} \|_{1,Q_L}^2
  + s \mathfrak{h}_k(s) \right), \nonumber
  \end{split}
\eea
for $s$ large enough, because the terms coming from the operators $R_k$ are absorbed by the terms in the left hand side.\\

The time dependence of $\tilde{c}_i$ implies to manage more terms at this part of the proof than in \cite{CLS15}. More precisely, we have to deal in the right hand side with the terms $\mathfrak{h}_k(s)$ for $k=0$ and $k=1$.
In light of \eqref{s10} this entails that
\bea
& & \sum_{j=0,1} \| e^{s \varphi(\cdot,0)} \nabla^j u_\chi^{(2)}(\cdot,0) \|_{0,\Omega_L}^2 \label{s13}\\
& \leq &
C \sum_{k=0}^3 \left( \| {\rm e}^{s \varphi} f_{c_\chi}^{(k)} \|_{0,Q_L}^2 + \| {\rm e}^{s \varphi} g_k \|_{0,\tilde{Q}_{\eps}}^2 + s^3 {\rm e}^{2 s \tilde{d}_\ell} \| u_\chi^{(k)} \|_{1,Q_L}^2 + s\mathfrak{h}_k(s) \right). \nonumber 
\eea

Further, we see that the first (resp., second) term of the sum in the right hand side of
\eqref{s13} is upper bounded up to some multiplicative constant, by $\sum_{j=0,1} \| {\rm e}^{s \varphi} \nabla^j c_\chi \|_{0,Q_L}^2$ (resp., ${\rm e}^{2 s \tilde{d}_\ell} ( \| u^{(k)} \|_{1,Q_L}^2 + 1)$), as in \cite{CLS15}.
From this and Lemma \ref{lm1} then follows for $s$ sufficiently large that

\bea
& &  C \, s  \sum_{j=0,1} \| e^{s \varphi(\cdot,0)} \nabla^j c_\chi \|_{0,\Omega_L}^2  \label{3.18} \\
& \leq &  \sum_{j=0,1} \| {\rm e}^{s \varphi} \nabla^j c_\chi \|_{0,Q_L}^2 + {\rm e}^{2 s \tilde{d}_\ell} +\sum_{k=0}^3 \left( s^3 {\rm e}^{2 s \tilde{d}_\ell} \| u^{(k)} \|_{1,Q_L}^2 + s \mathfrak{h}_k(s)\right).  \nonumber
\eea

Finally, the end of the proof is similar to \cite{CLS15}, adding in the right hand side the terms in the sum for $k=0$ and $k=1$. But, these additive terms are already embedded in the norm used in our final result, and we get the stability inequality of theorem \ref{T.1}.

 \begin{figure}
 \begin{center}
   \begin{tabular}{cc}
{\includegraphics[scale=0.17,  trim = 8.0cm 6.0cm 8.0cm 6.0cm,clip=true,]{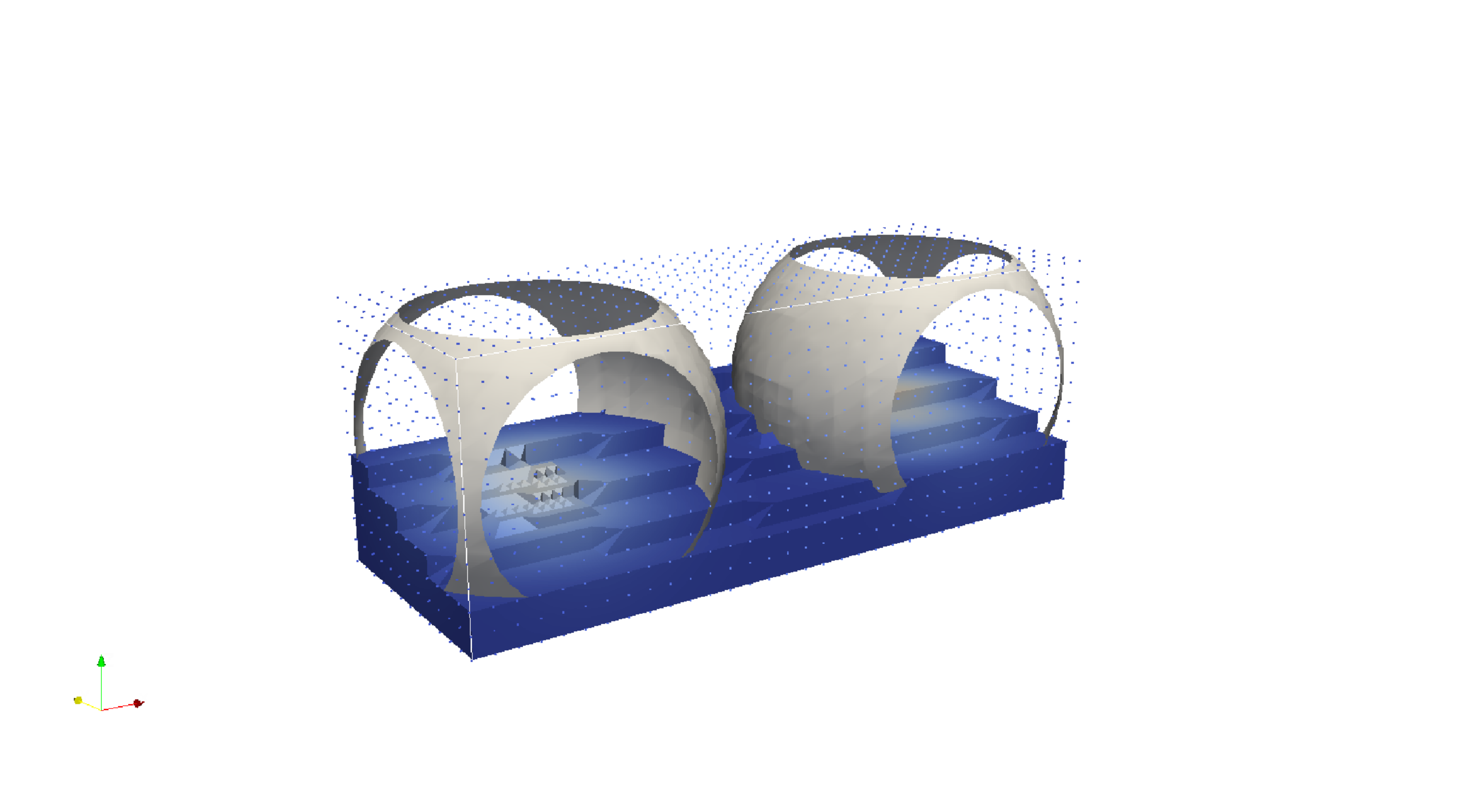}} &
  {\includegraphics[scale=0.17,  trim = 8.0cm 6.0cm 8.0cm 6.0cm, clip=true,]{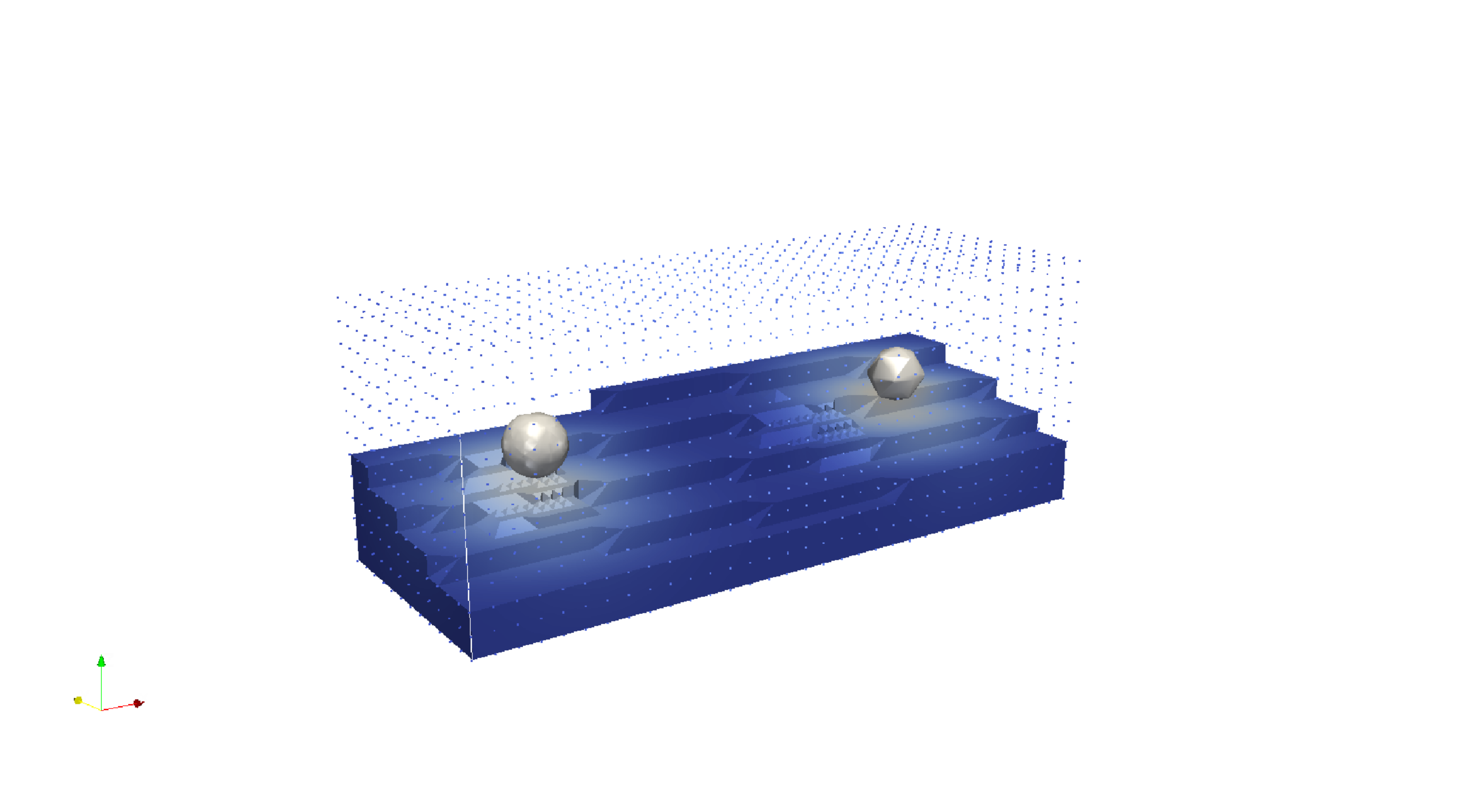}}  \\
a) $c(x)=1.2$ &  b) $c(x)=5.5$ \\
\end{tabular}
 \end{center}
 \caption{ \protect\small \emph{Slices of the exact Gaussian function $c(x)$ in
   $\Omega_{FEM}$ given by (\ref{2gaussians}).  }}
 \label{fig:exact_gaus}
 \end{figure}

 \begin{figure}
 \begin{center}
   \begin{tabular}{cc}
{\includegraphics[scale=0.2, trim = 8.0cm 6.0cm 8.0cm 6.0cm, clip=true,]{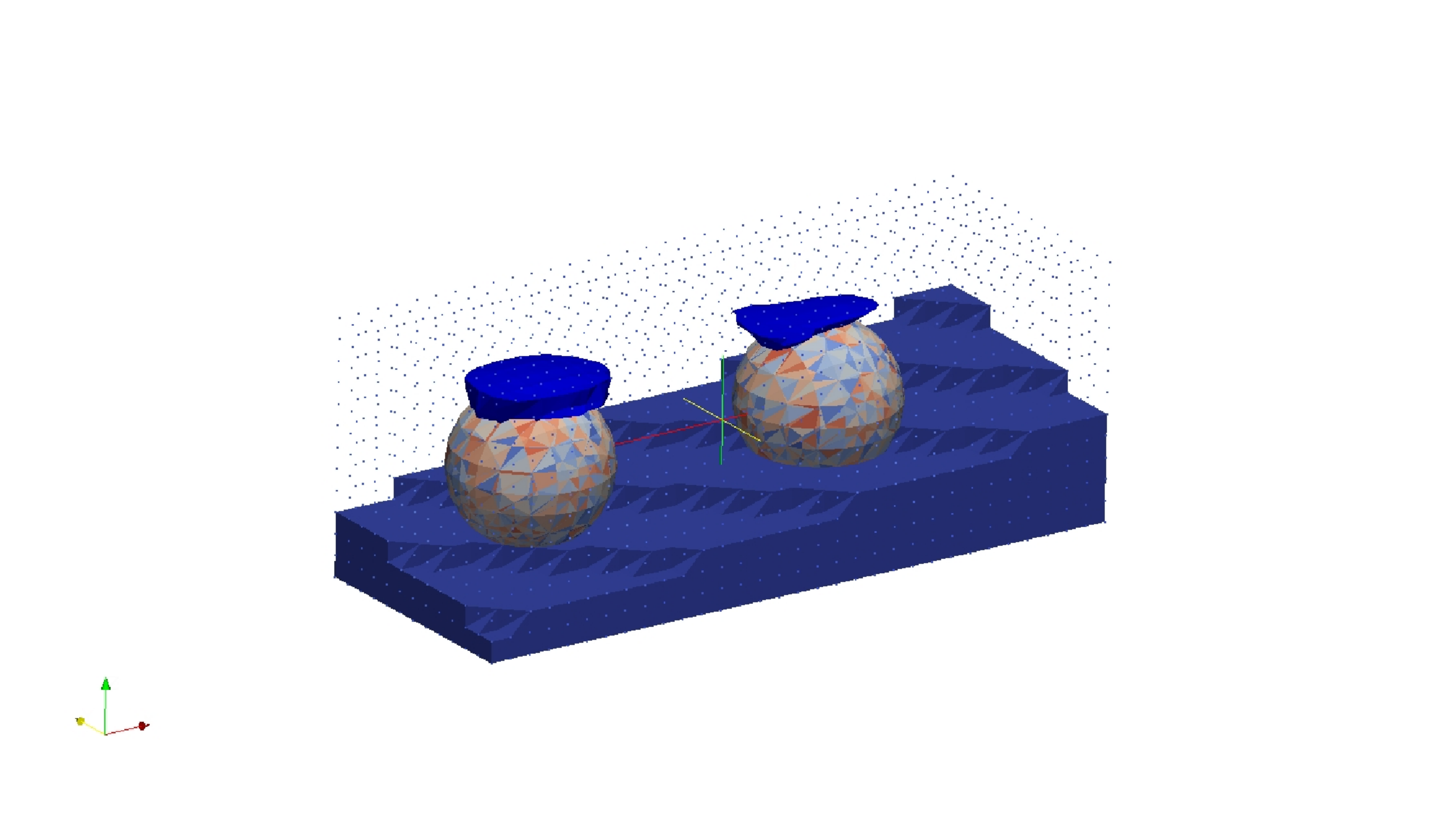}} &
  {\includegraphics[scale=0.2, trim = 8.0cm 6.0cm 8.0cm 6.0cm, clip=true,]{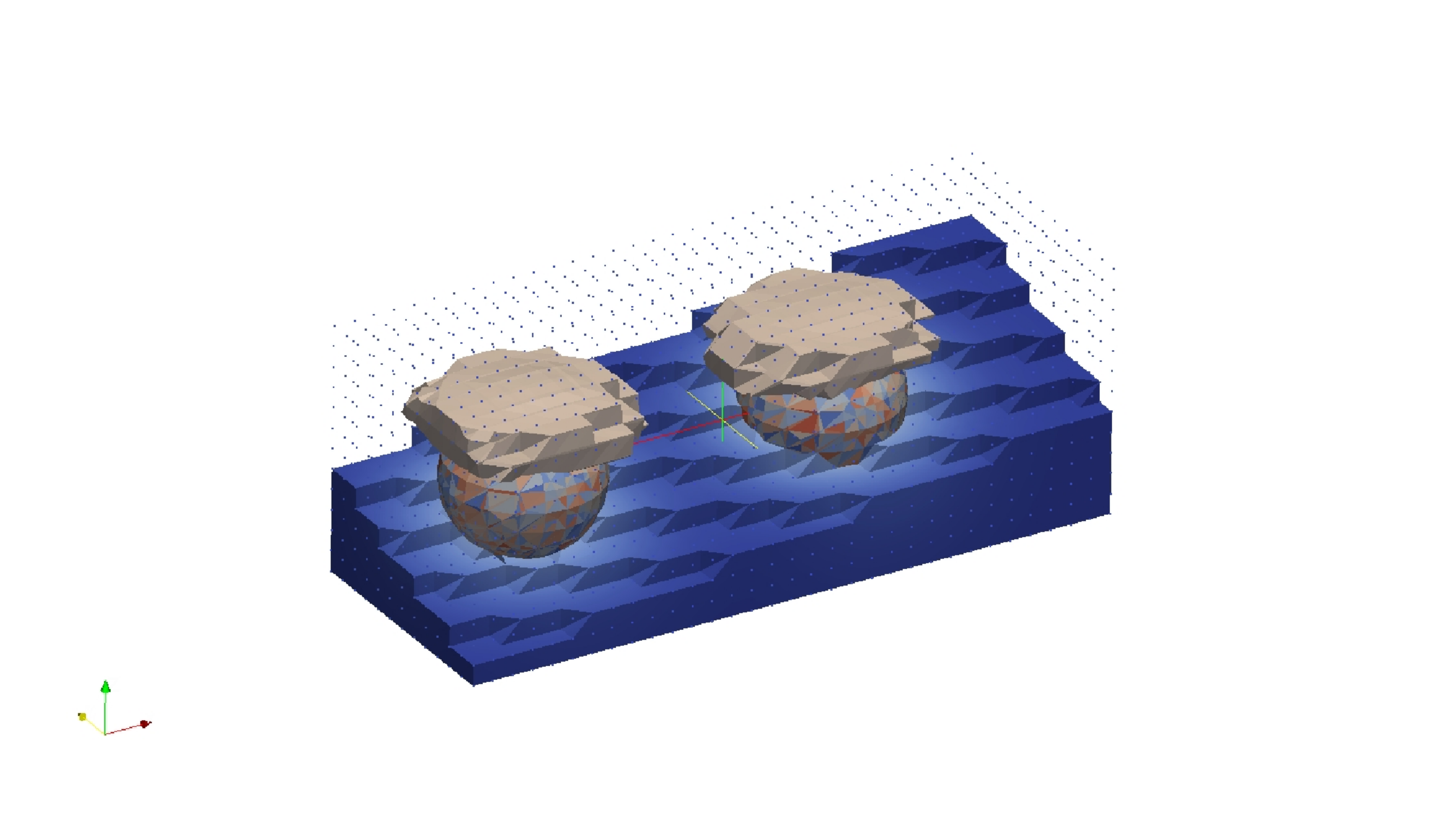}} \\
a) $\sigma=3\%$ & b)  $\sigma= 10\%$
\end{tabular}
 \end{center}
 \caption{ \protect\small \emph{Reconstructions
   obtained in Test 1 for different noise level  $\sigma$ in data. All figures are visualized for $c(x)=3.5$. }}
 \label{fig:rec_test1}
 \end{figure}

\section{Numerical Studies}
\label{sec:Numer-Simul}

In this section, we present numerical simulations of the
reconstruction of unknown function $c(x)$ of the equation (\ref{S1})
using the  domain decomposition method of \cite{BAbsorb}.

To do that we decompose $\Omega$ into two subregions $\Omega_{FEM}$
and $\Omega_{FDM}$ such that $\Omega = \Omega_{FEM} \cup
\Omega_{FDM}$, and $\Omega_{FEM} \cap \Omega_{FDM} = \emptyset$.  In
$\Omega_{FEM}$ we will use the finite element method (FEM) and in
$\Omega_{FDM}$ - the finite difference method (FDM). 
 The boundary $\partial \Omega$ of the domain $\Omega$ is such
 that $\partial \Omega =\partial _{1} \Omega \cup \partial _{2} \Omega
 \cup \partial _{3} \Omega$ where $\partial _{1} \Omega$ and $\partial
 _{2} \Omega$ are, respectively, front and back sides of $\Omega$, and
 $\partial _{3} \Omega$ is the union of left, right, top and bottom
 sides of this domain. We will collect time-dependent observations
 $S_T := \partial_1 \Omega \times (0,T)$ at the backscattering side
 $\partial_1 \Omega$ of $\Omega$.  We
 also define $S_{1,1} := \partial_1 \Omega \times (0,t_1]$, $S_{1,2}
:= \partial_1 \Omega \times (t_1,T)$, $S_2 := \partial_2 \Omega \times
(0, T)$ and $S_3 := \partial_3 \Omega \times (0, T)$.

Our model problem used in computations is  following:
\begin{equation}\label{model1}
\begin{split}
\frac{\partial^2 u}{\partial t^2}  -  \nabla \cdot ( \tilde{c} \nabla  u)   &= 0~ \mbox{in}~~ \Omega_T, \\
  u(x,0) = \theta_0(x), ~~~u_t(x,0) &= 0~ \mbox{in}~~ \Omega,     \\
\partial _{n} u& = f(t) ~\mbox{on}~ S_{1,1},
\\
\partial _{n}  u& =-\partial _{t} u ~\mbox{on}~ S_{1,2},
\\
\partial _{n} u& =-\partial _{t} u~\mbox{on}~ S_2, \\
\partial _{n} u& =0~\mbox{on}~ S_3.\\
\end{split}
\end{equation}
 In (\ref{model1}) the function $f(t)$ is the single direction of a
 plane wave which is initialized at $\partial_1 \Omega$ in time
 $T=[0,3.0]$ and is defined as
 \begin{equation}\label{f}
 \begin{split}
 f(t) =\left\{ 
 \begin{array}{ll}
 \sin \left( \omega t \right) ,\qquad &\text{ if }t\in \left( 0,\frac{2\pi }{\omega }
 \right) , \\ 
 0,&\text{ if } t>\frac{2\pi }{\omega }.
 \end{array}
 \right. 
 \end{split}
 \end{equation}
We  initialize initial condition $\theta_0(x)$ at backscattered side $\partial_1 \Omega$ as 
\begin{equation}\label{initcond}
\begin{split}
u(x,0) &= f_0(x)={\rm e}^{-(x_1^2 + x_2^2 + x_3^3)}  \cdot \cos  t|_{t=0} = {\rm e}^{-(x_1^2 + x_2^2 + x_3^3)}.
\end{split}
\end{equation}

We assume that the functions $c(x)=1$ and $c_0(x,t)=0$ inside
$\Omega_{FDM}$.  The goal of our numerical tests is to reconstruct
 a smooth
function $c(x)$ only inside $\Omega_{FEM}$ which we define as
\begin{equation}\label{2gaussians}
\begin{split}
c(x) &= 1.0 + 5.0 \cdot {\rm e}^{-((x_1 -0.5)^2/0.2 + {x_2}^2/0.2 + {x_3}^2/0.2 )} \\
&+ 5.0 \cdot {\rm e}^{-((x_1 +1)^2/0.2 + {x_2}^2/0.2 + {x_3}^2/0.2 )}.
\end{split}
\end{equation}
We also assume that the function $c_0(x,t)$ is known inside
$\Omega_{FEM}$, and we define this function  as
\begin{equation}\label{timedepfunc}
\begin{split}
  c_0(x,t) &= 0.01 ~ \cos t \cdot {\rm e}^{-({x_1}^2/0.2 + {x_2}^2/0.2 + {x_3}^2/0.2 )}.
\end{split}
\end{equation}
Figure \ref{fig:exact_gaus} presents slices of the exact function
$c(x)$ given by (\ref{2gaussians}) for $c(x)=1.2$ and $c(x)=5.5$, correspondingly, and
Figure \ref{fig:rec_test1} presents isosurfaces of the exact function
$\tilde{c}$ in the problem (\ref{model1}) for $\tilde{c}=3.5$ at
different times.  Numerical tests of \cite{BAbsorb} show that the best
reconstruction results for the space-dependent function $c(x)$ and for
$c_0=0$ in $\Omega$ are obtained for $\omega = 40$ in (\ref{f}), and
we take $\omega = 40$ in (\ref{f}) in all our tests.

 \begin{figure}
 \begin{center}
 \begin{tabular}{cc}
{\includegraphics[scale=0.23, clip=true,]{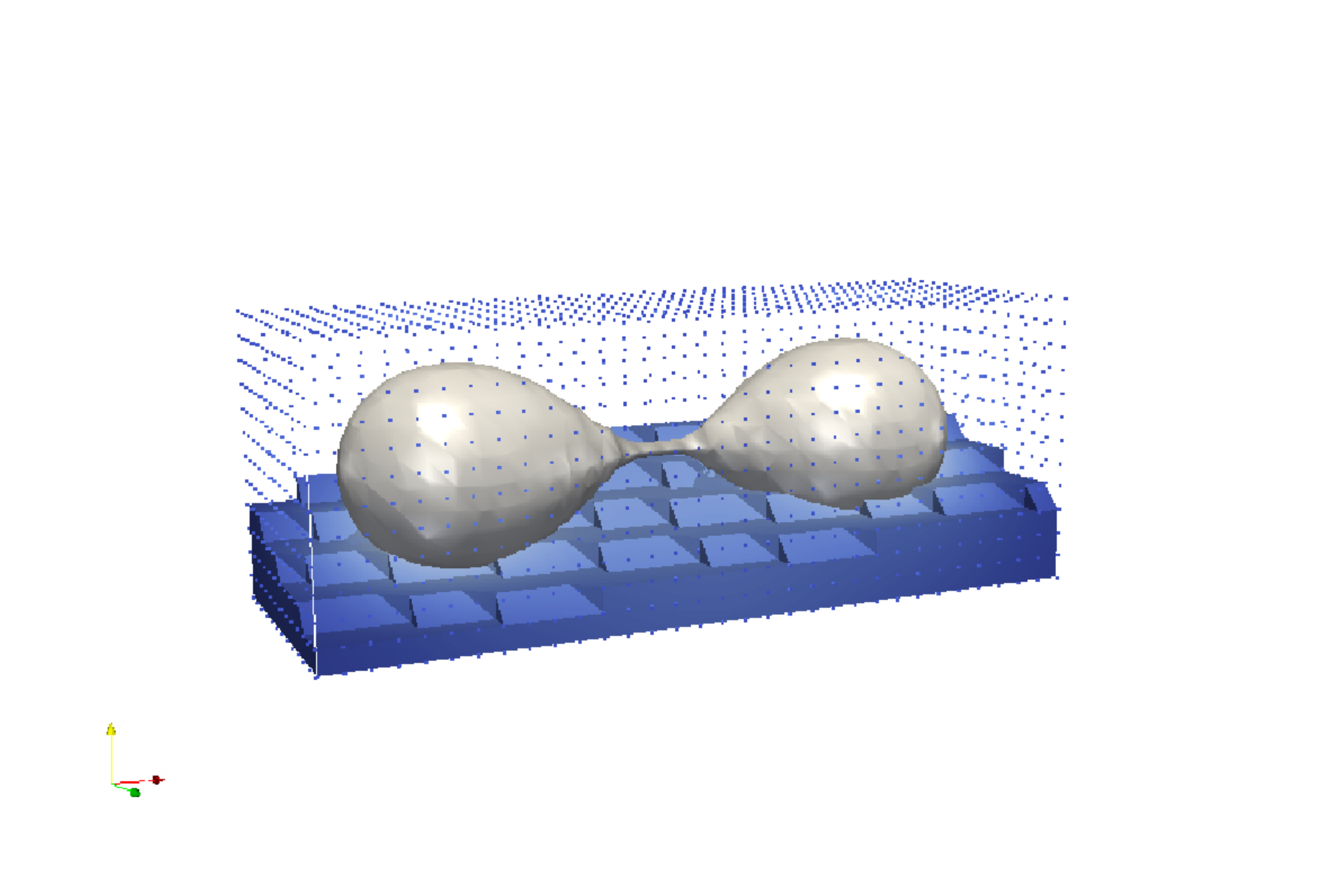}} &
  {\includegraphics[scale=0.23,  clip=true,]{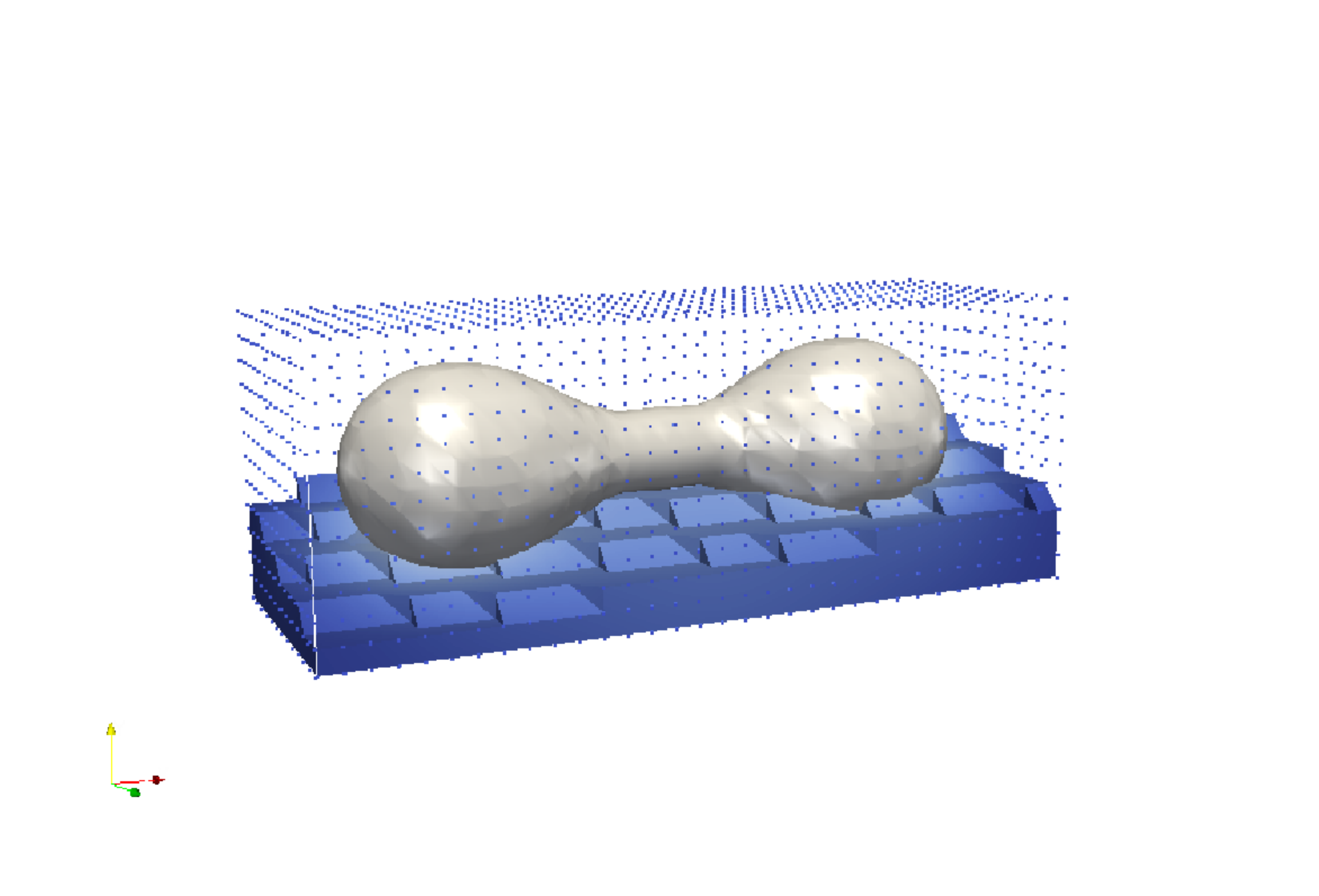}} 
 \\
 t= 1.2 & t = 1.8 \\
{\includegraphics[scale=0.23, clip=true,]{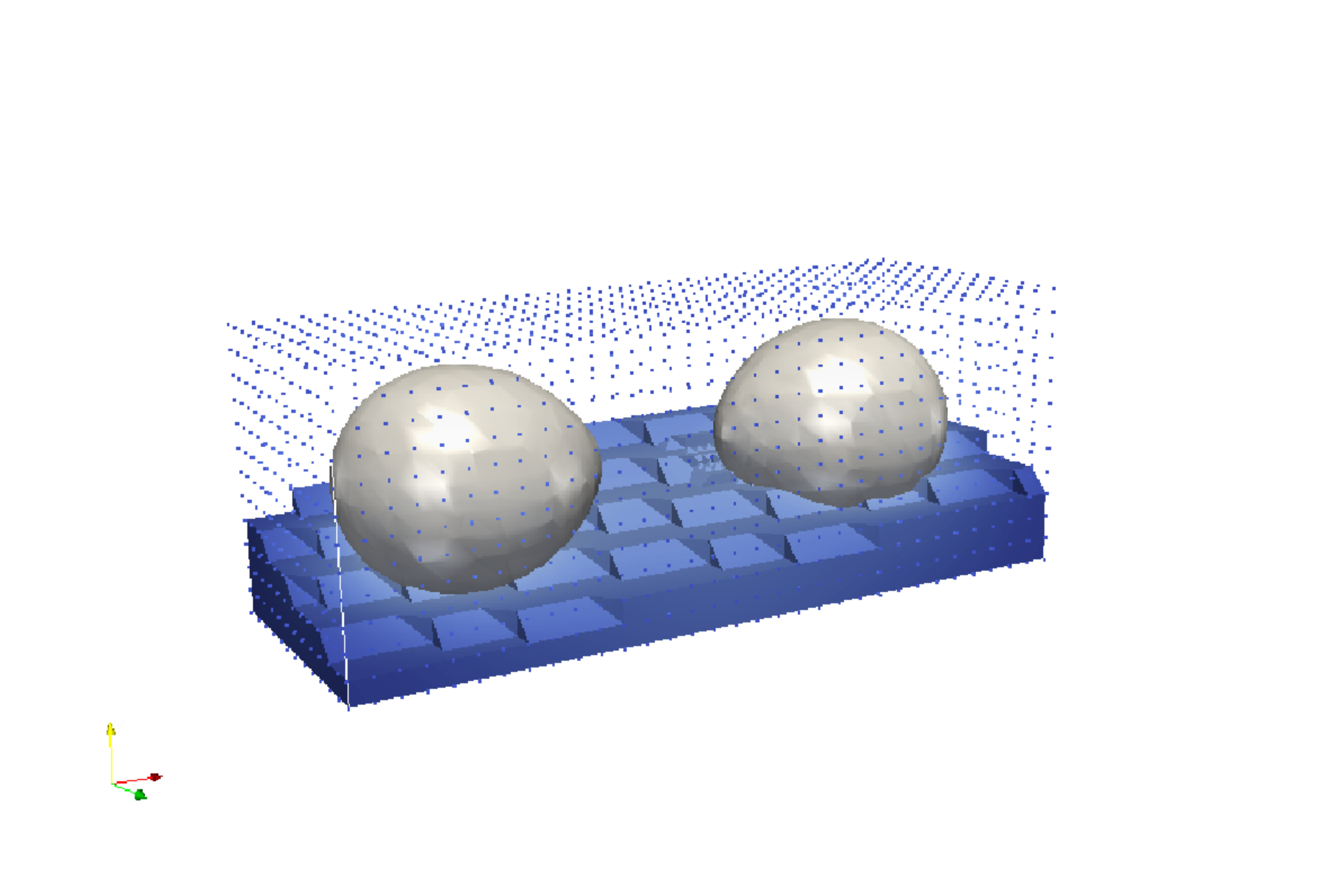}} &
  {\includegraphics[scale=0.23,  clip=true,]{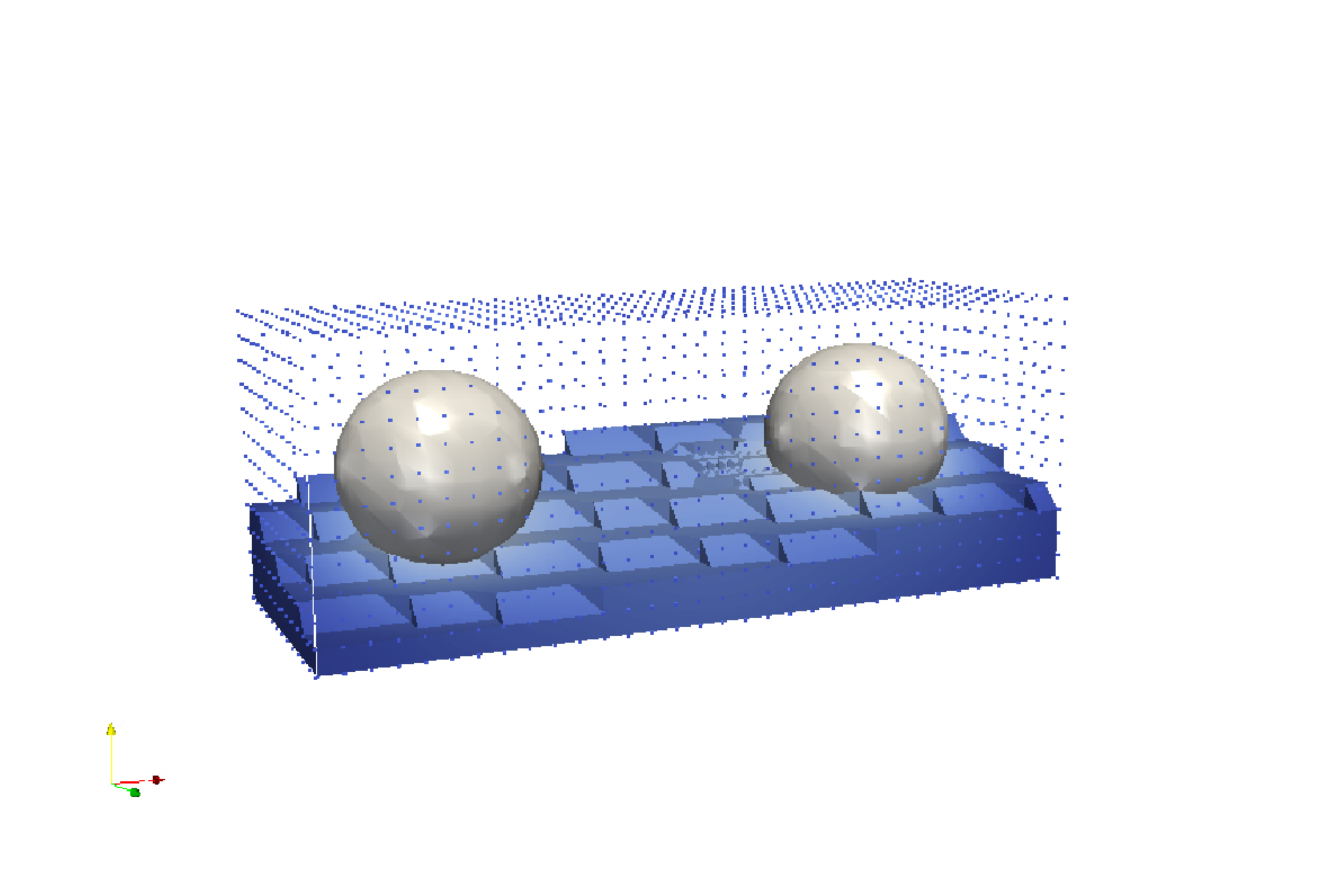}} 
 \\
t=2.4 & t = 2.7 \\
\end{tabular}
 \end{center}
 \caption{ \protect\small \emph{ Slices of the exact space and
     time-dependent function $\tilde{c}$ for $\tilde{c}=3.5$ at
     different times.} }
 \label{fig:exact_functime_b}
 \end{figure}

We introduce dimensionless spatial
 variables $x^{\prime}= x/\left(1m\right)$ 
and define $\Omega_{FEM}$  and $\Omega_{FDM}$ as 
 the following dimensionless computational domains:
 \begin{equation*}
 \Omega_{FEM} = \left\{ x= (x_1,x_2,x_3) \in (
 -1.6,1.6) \times (-0.6,0.6) \times (-0.6,0.6) \right\},
 \end{equation*}
 \begin{equation*}
 \Omega = \left\{ x= (x_1,x_2,x_3) \in (
 -1.8,1.8) \times (-0.8,0.8) \times (-0.8,0.8) \right\} .
 \end{equation*}
  We choose the mesh size $h=0.1$ in the overlapping layers between
  $\Omega_{FEM}$ and $\Omega_{FDM}$ as well as in the computations of
  the inverse problem. However, we have generated our backscattered
  data using the several times locally refined mesh $\Omega_{FEM}$,
  and in a such way we avoid problem with variational crime.  To
  generate backscattered data we solve the model problem
  (\ref{model1}) in time $T=[0,3.0]$ with the time step $\tau=0.003$
  which satisfies to the CFL condition \cite{CFL67}, and supply
  simulated backscattered data by additive noise $\sigma=3\%, 10\%$ at
  $\partial_1 \Omega$. Similarly with \cite{BAbsorb, BNAbsorb} in all
  our computations we choose constant regularization parameter $\gamma
  =0.01$ because it gives smallest relative error in the
  reconstruction of the function $c(x)$.  Different techniques for the
  computation of a regularization parameter are presented in works
  \cite{BKS,Engl,tikhonov}, and checking of performance of these
  techniques for the solution of our inverse problem can be challenge
  for our future research.

 We
 also assume that the reconstructed function $c(x)$ belongs to the set
 of admissible parameters
 \begin{equation}\label{admpar}
 \begin{split}
  M_{c} \in \{c\in C(\overline{\Omega })|1\leq c(x)\leq 10\}.\\
 \end{split}
 \end{equation}
To get final images of our reconstructed function $c(x)$ we use a
post-processing procedure which is the same as in \cite{BAbsorb,
  BNAbsorb}.

\begin{table}[tbp] 
{\footnotesize Table 1. \emph{Computational results of the
    reconstructions  together with computational errors
    in the maximal contrast of $c(x)$ in percents. Here, $\overline{N}$ is the
    final number of iteration in the conjugate gradient method.}}  \par
\vspace{2mm}
\centerline{
\begin{tabular}{|c|c|}
 \hline
   $\sigma=3\%$ &  $\sigma = 10\%$ 
 \\
 \hline
\begin{tabular}{l|l|l|l} \hline
Case & $\max_{\Omega_{FEM}} c_{\overline{N}}$ &  error, \% & $\overline{N}$  \\ \hline
Test 1 & 6.66 & 11 & 13   \\
Test 2 & 4.1  & 32  & 9    \\
\end{tabular}
 & 
\begin{tabular}{l|l|l|l} \hline
Case & $\max_{\Omega_{FEM}} c_{\overline{N}}$ &    error, \% & $\overline{N}$  \\ \hline
Test 1 & 8.11 & 35  & 15   \\
Test 2 & 5.58 &   7  & 10   \\
\end{tabular} 
\\
\hline
\end{tabular}}
\end{table}

 \begin{figure}
 \begin{center}
 \begin{tabular}{cc}
 {\includegraphics[scale=0.2, trim = 8.0cm 6.0cm 8.0cm 6.0cm,
     clip=true,]{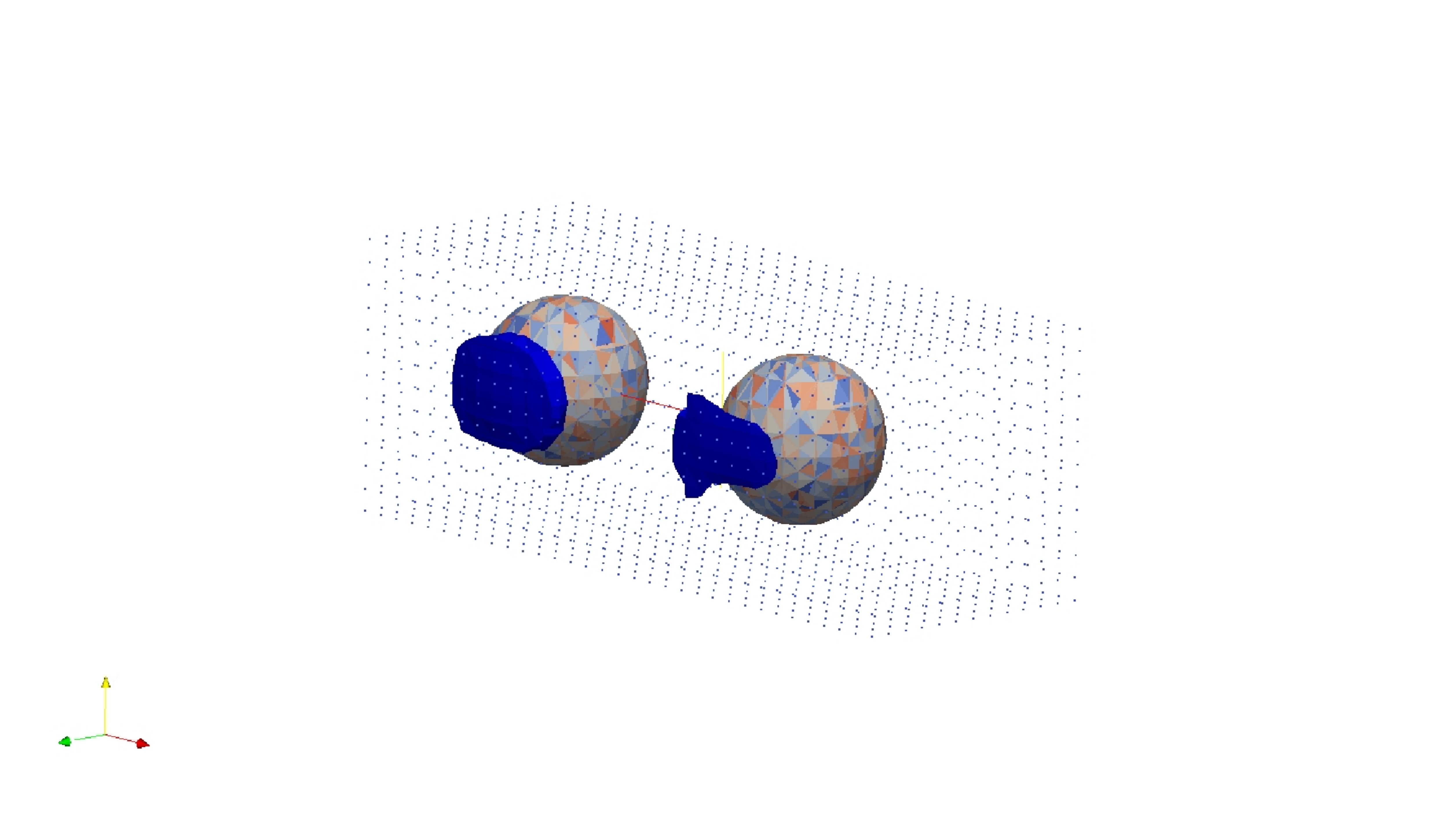}} &
 {\includegraphics[scale=0.2, trim = 8.0cm 6.0cm 8.0cm 6.0cm,
     clip=true,]{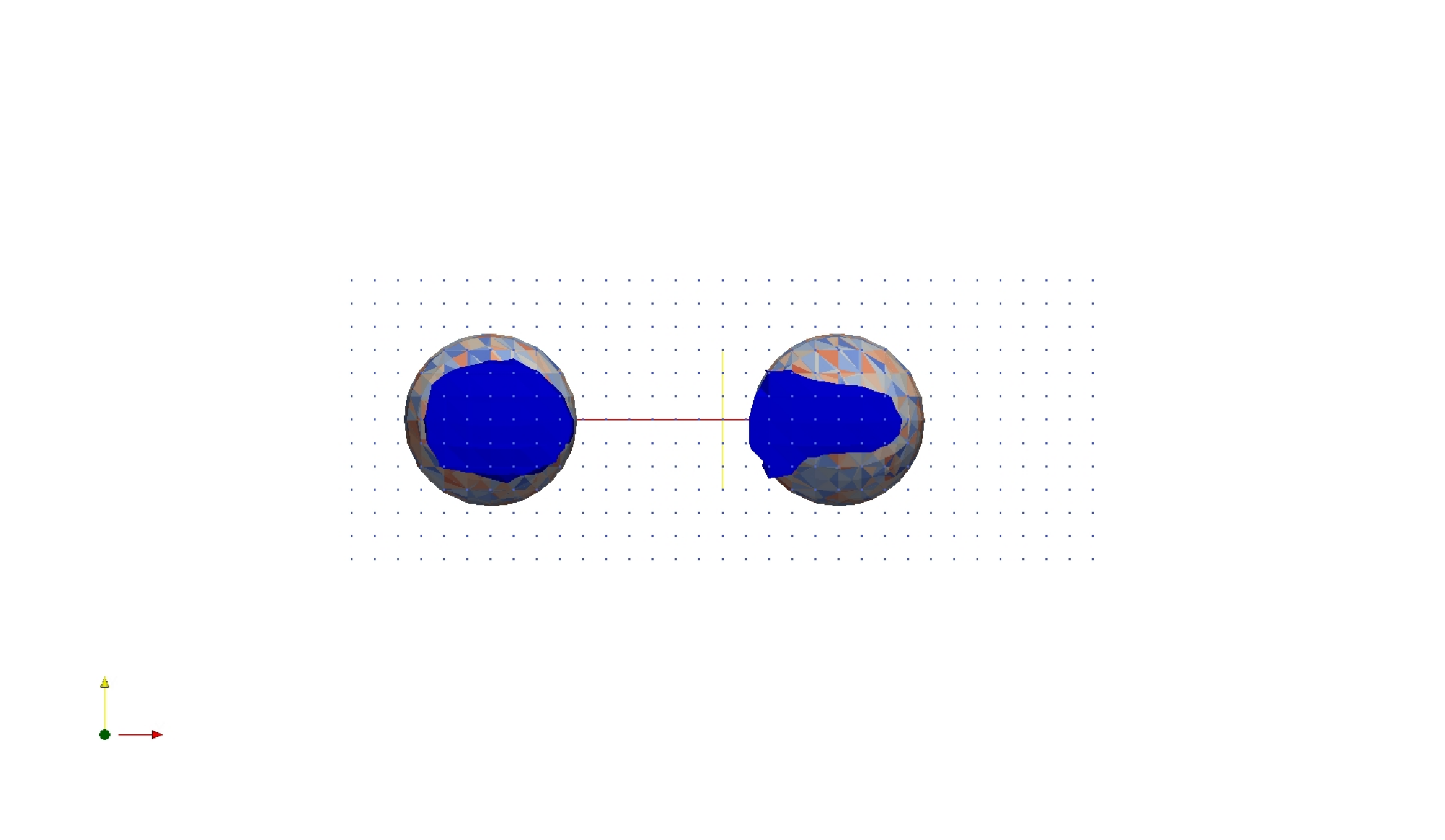}} \\
 prospect view & $x_1 x_2$ view \\
 {\includegraphics[scale=0.2, trim =  8.0cm 6.0cm 8.0cm 6.0cm,
     clip=true,]{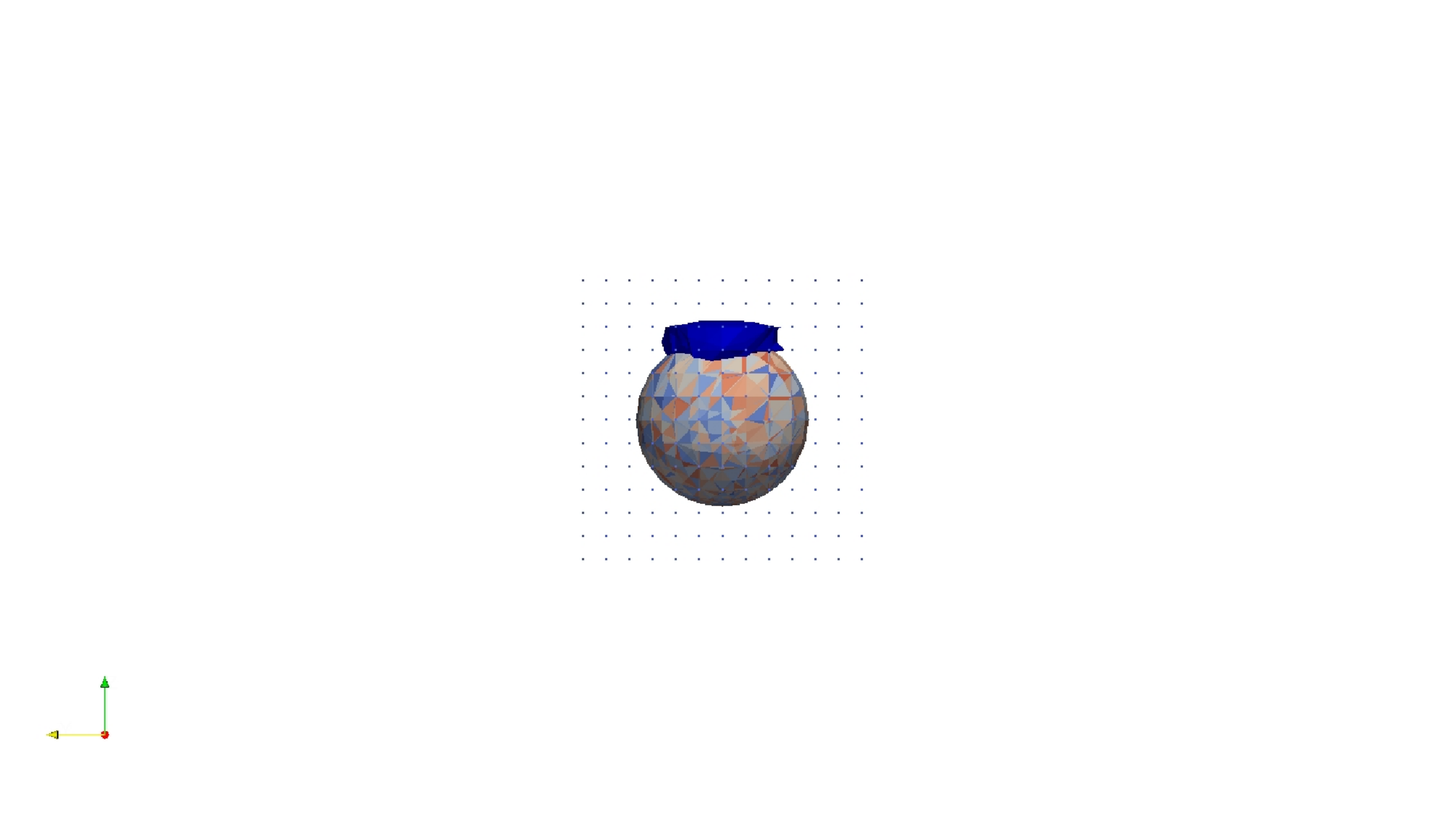}} &
 {\includegraphics[scale=0.2, trim = 8.0cm 6.0cm 8.0cm 6.0cm,
     clip=true,]{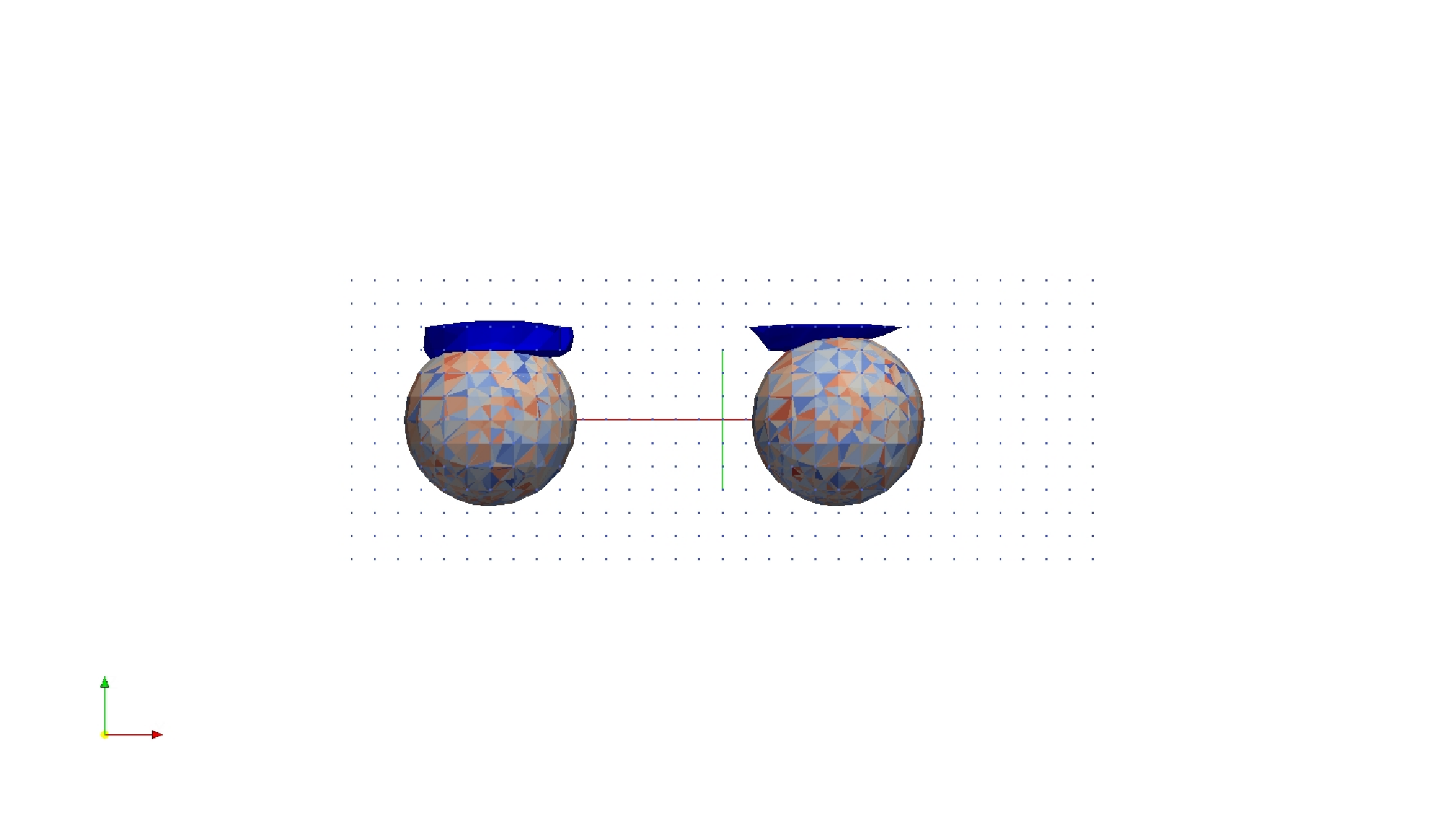}} \\ $x_2 x_3$ view & $x_3 x_1$ view
 \\
\end{tabular}
 \end{center}
 \caption{ \protect\small \emph{  Test 1. Reconstruction of $c(x)$  with $\max_{\Omega_{FEM}} c(x) = 6.66 $ for
   $\omega=40$ in (\ref{f}) with noise level $\sigma=3\%$. The  spherical wireframe of the isosurface
     with exact value of the function (\ref{2gaussians}),
  which corresponds to the value of the reconstructed $c=
     0.7\max_{\Omega_{FEM}} c(x)$.} }
 \label{fig:test1noise3}
 \end{figure}

 \begin{figure}
 \begin{center}
 \begin{tabular}{cc}
 {\includegraphics[scale=0.2, trim = 8.0cm 6.0cm 8.0cm 6.0cm,
     clip=true,]{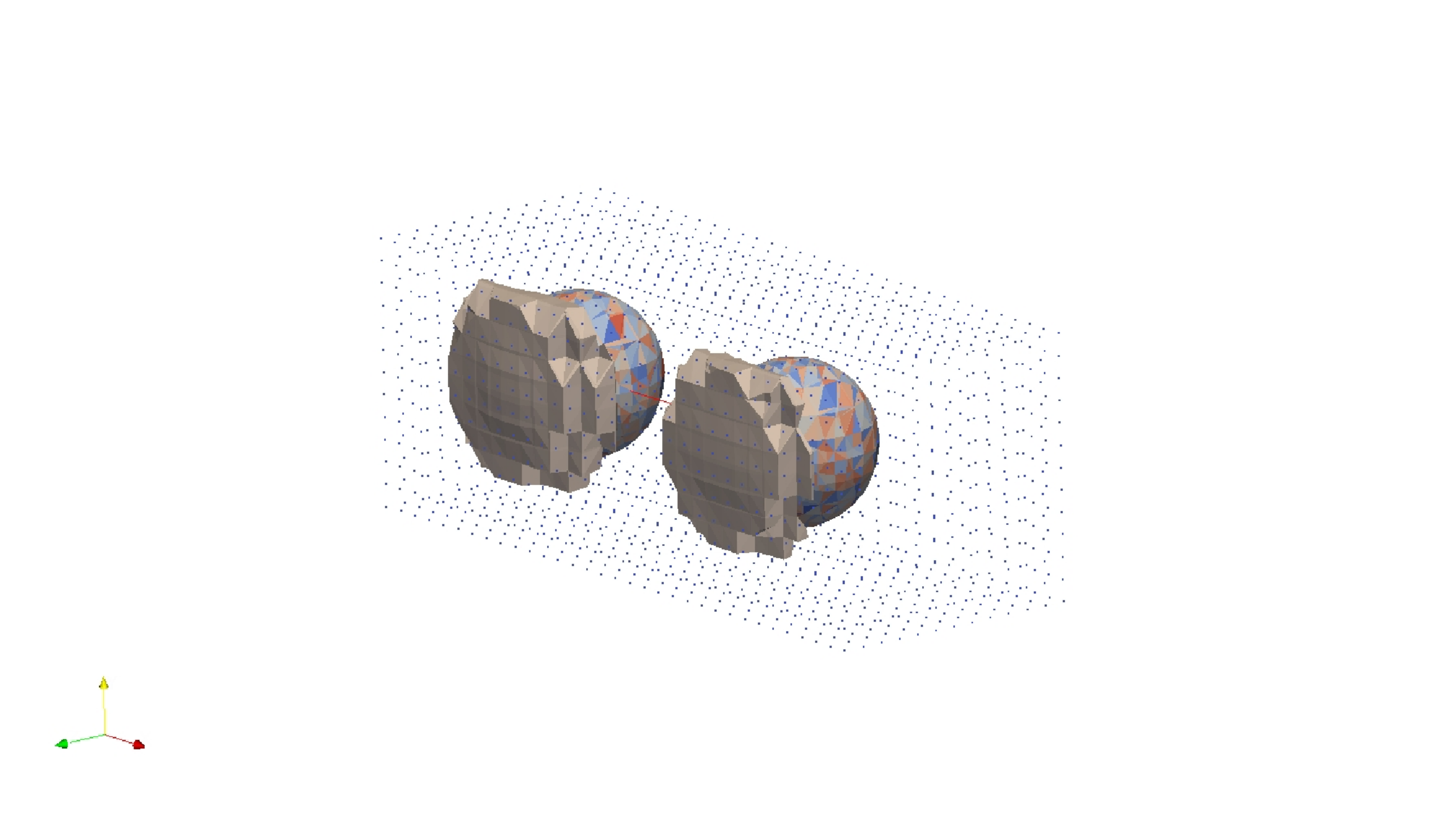}} &
 {\includegraphics[scale=0.2, trim = 8.0cm 6.0cm 8.0cm 6.0cm,
     clip=true,]{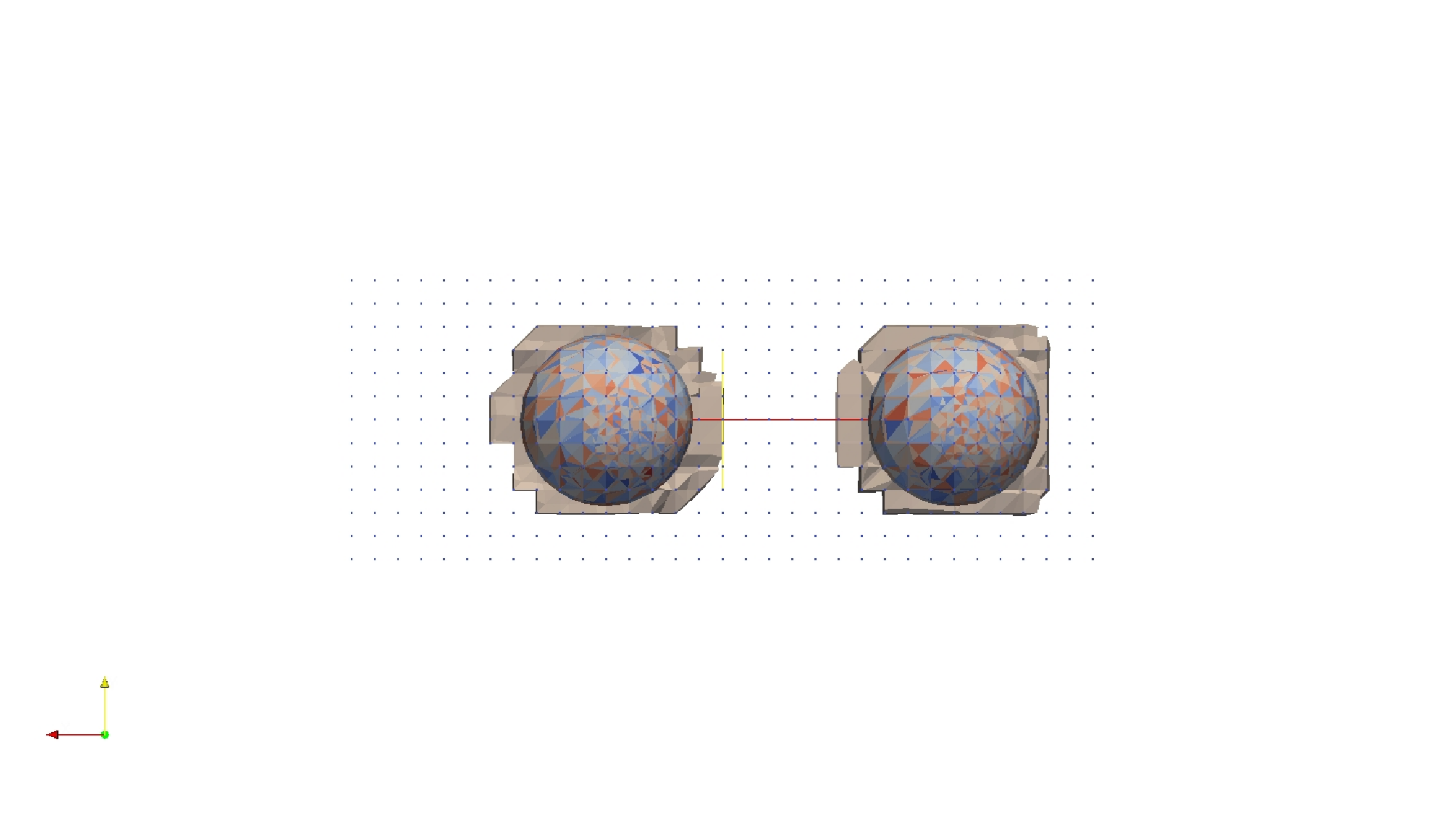}} \\
 prospect view & $x_1 x_2$ view \\
 {\includegraphics[scale=0.2, trim = 8.0cm 6.0cm 8.0cm 6.0cm,
     clip=true,]{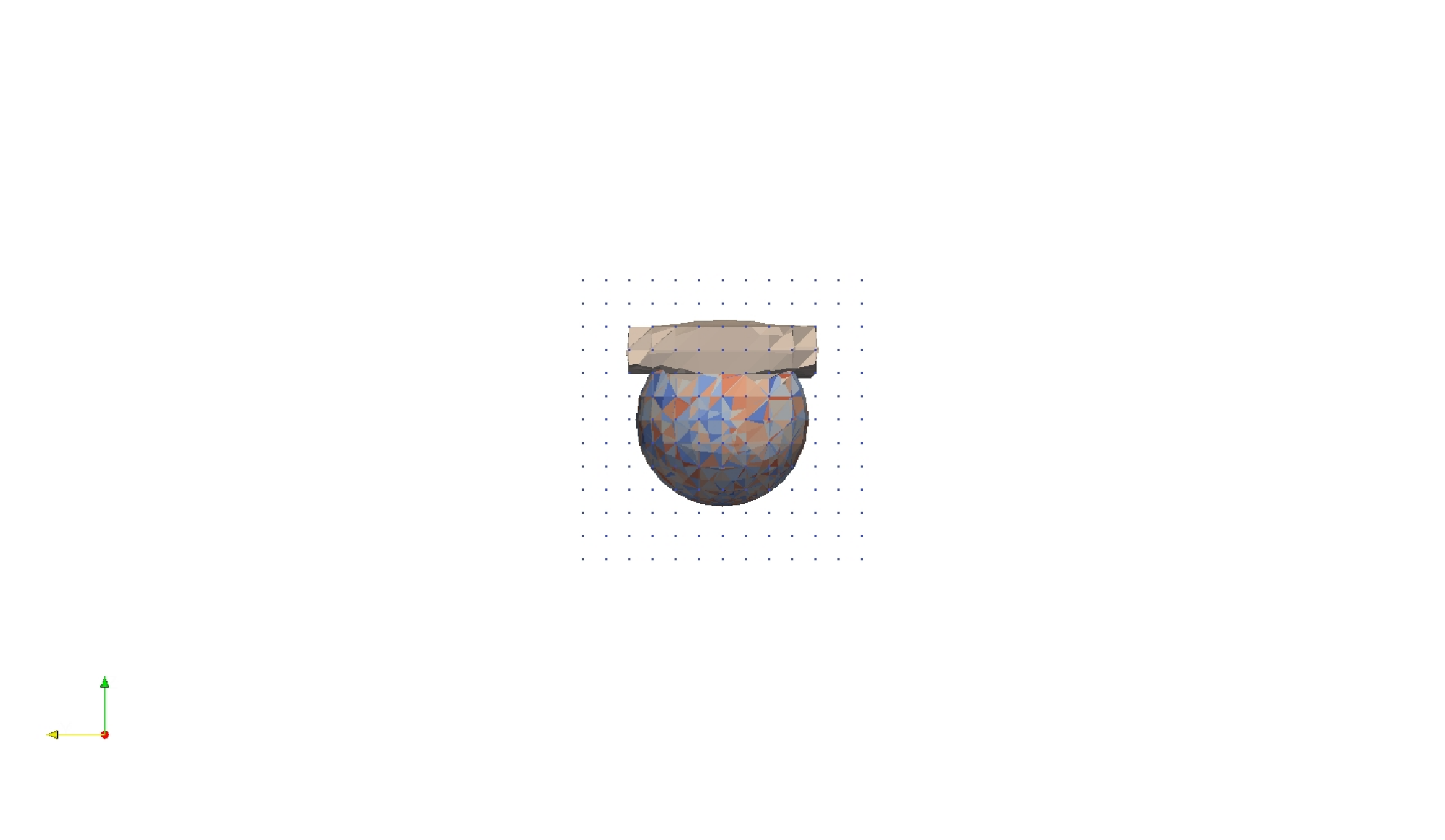}} &
 {\includegraphics[scale=0.2, trim = 8.0cm 6.0cm 8.0cm 6.0cm,
     clip=true,]{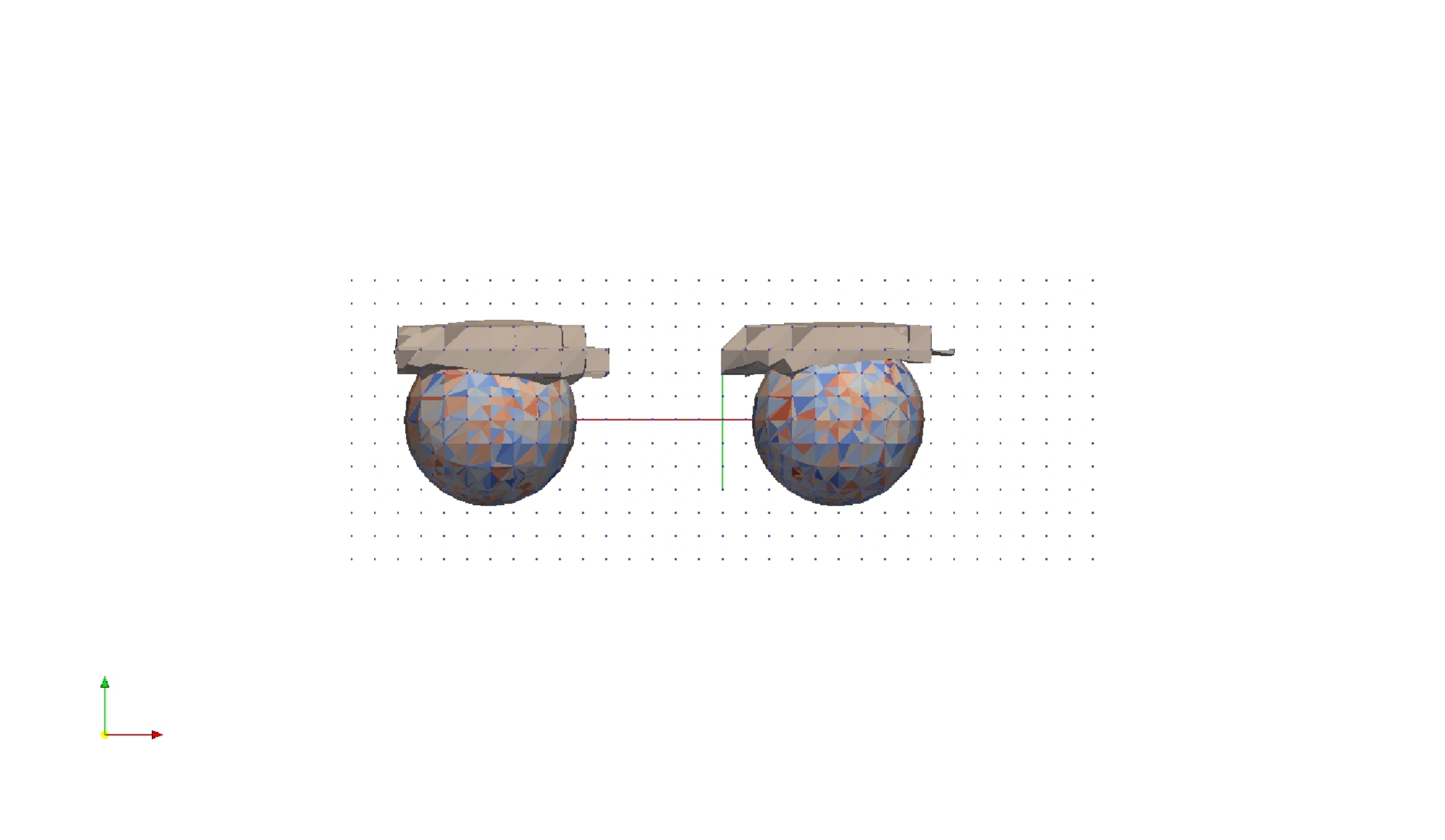}} \\ $x_2 x_3$ view & $x_3 x_1$ view
 \\
\end{tabular}
 \end{center}
 \caption{ \protect\small \emph{ Test 1. Reconstruction of $c(x)$ with
     $\max_{\Omega_{FEM}} c(x) = 8.11$ for $\omega=40$ in (\ref{f})
     with noise level $\sigma=10\%$. The spherical wireframe of the
     isosurface with exact value of the function (\ref{2gaussians}),
     which corresponds to the value of the reconstructed $c=
     0.7\max_{\Omega_{FEM}} c(x)$.} }
 \label{fig:test1noise10}
 \end{figure}

 \subsection{Test 1}

\label{sec:test1}

In this section we present numerical results of the reconstruction of  the
function $c(x)$ given by (\ref{2gaussians}), see Figure
\ref{fig:exact_gaus}-a), b), assuming, that the function $c_0(x,t)=0$.
In this case we obtain results similar to ones of \cite{BAbsorb, BNAbsorb}.

Figures \ref{fig:test1noise10} display results of the reconstruction
of function given by (\ref{2gaussians}) with additive noise
$\sigma=10\%$. Quite similar results are obtained for $\sigma=3\%$,
see Figure \ref{fig:test1noise3}. We observe that the location of the
maximal value of the function (\ref{2gaussians}) is imaged
correctly. It follows from Figure \ref{fig:test1noise3} and Table 1
that the imaged contrast in this function is
$6.66:1=\max_{\Omega_{FEM}} c_{13}:1 $, where $n:=\overline{N}=13$ is
the final iteration in the conjugate gradient method. Similar
observation we made using the Figure \ref{fig:test1noise10} and Table
1 where the imaged contrast is $8.11:1=\max_{\Omega_{FEM}} c_{15}:1 $,
$n:=\overline{N}=15$. However, from these figures we also observe that
because of the data post-processing procedure \cite{BAbsorb} the
values of the background of function (\ref{2gaussians}) are not
reconstructed but are smoothed out. Thus, we are able to reconstruct
only maximal values of the function (\ref{2gaussians}).  Comparison of
Figures \ref{fig:exact_gaus}-c), d), \ref{fig:test1noise3},
\ref{fig:test1noise10} with Figure~\ref{fig:exact_gaus}-a), b) reveals
that it is desirable to improve shape of the function
(\ref{2gaussians}) in $x_3$ direction.

 \subsection{Test 2}

\label{sec:test2}

 \begin{figure}
 \begin{center}
 \begin{tabular}{cc}
 {\includegraphics[scale=0.22,  clip=true,]{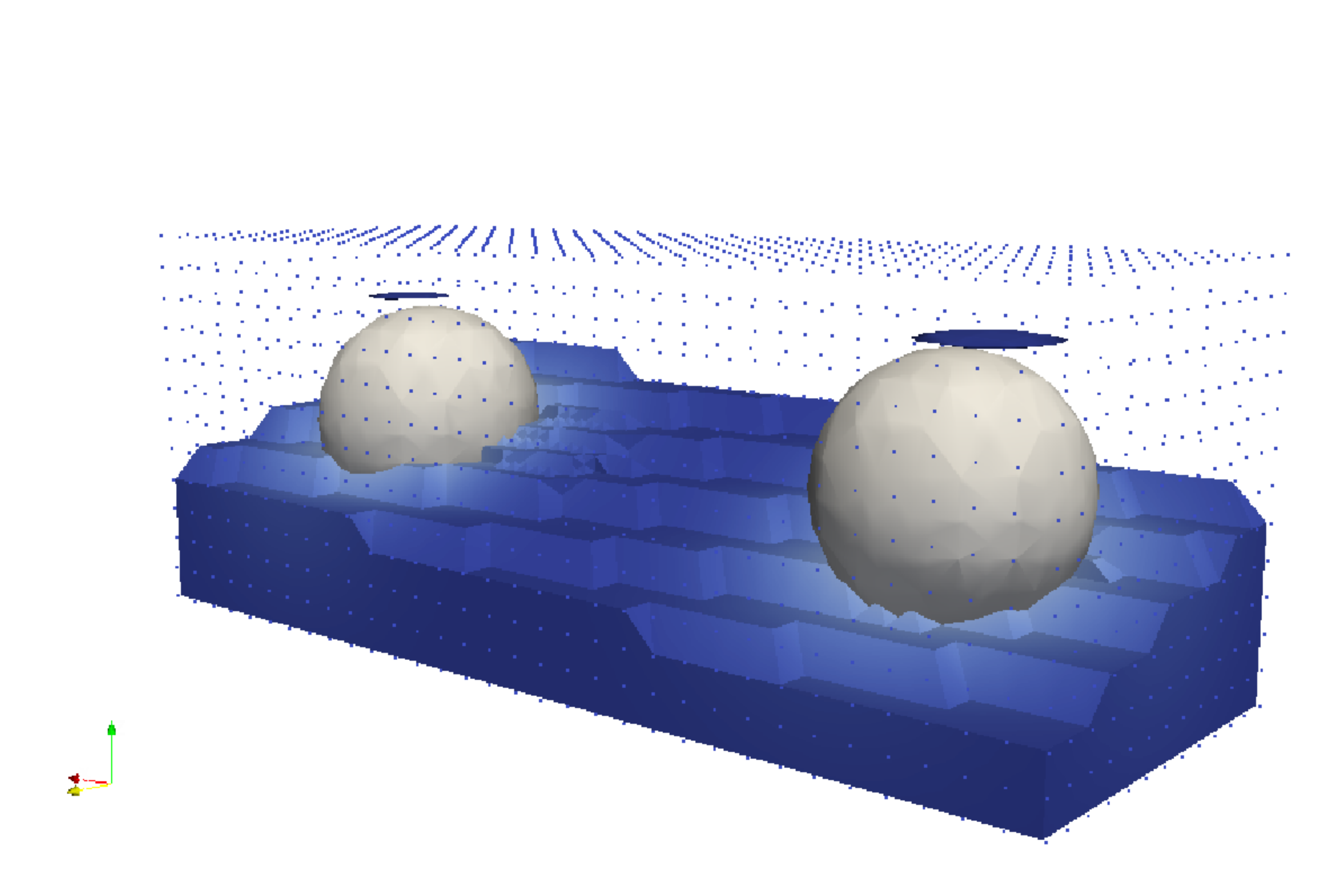}} &
 {\includegraphics[scale=0.22,  clip=true,]{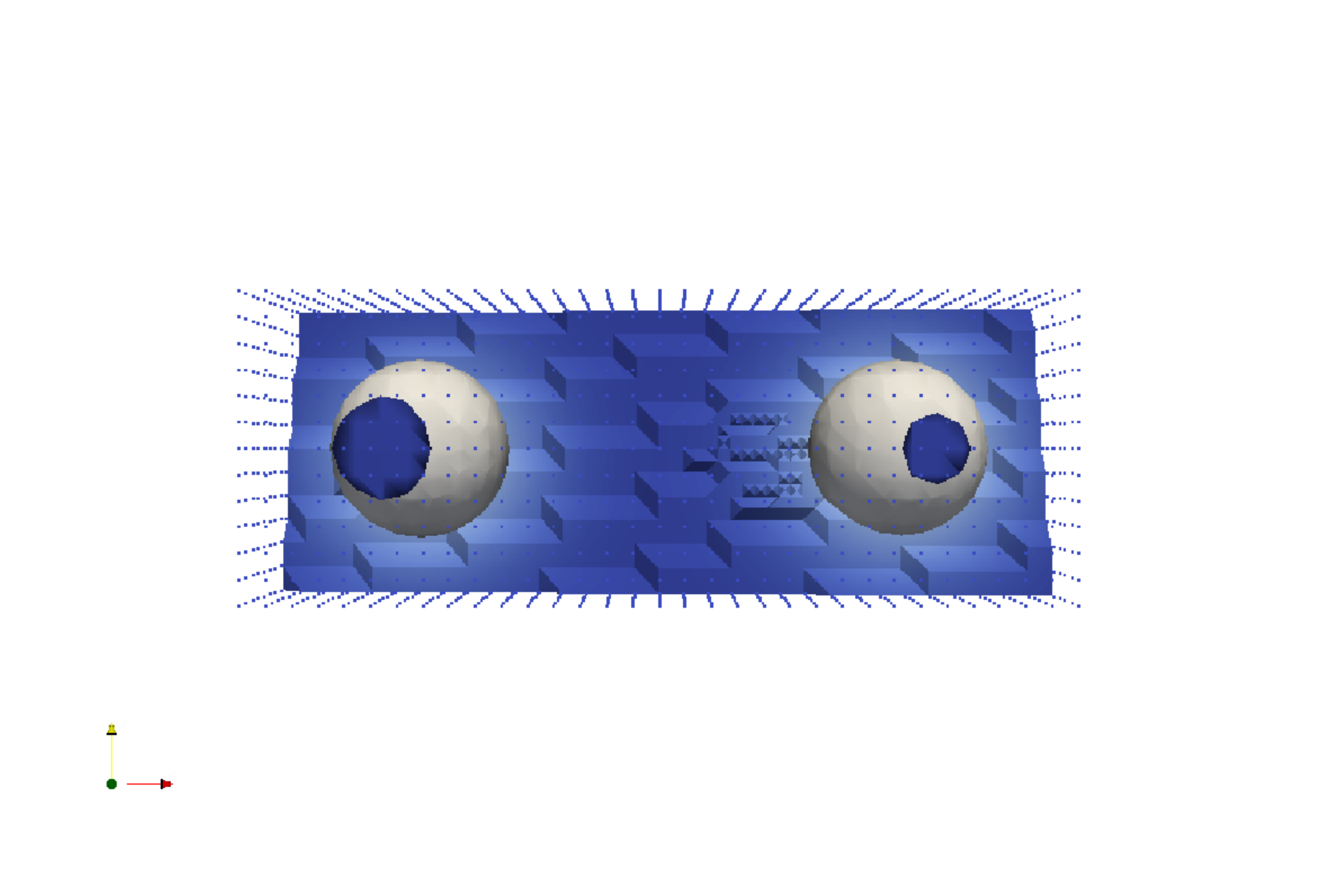}} \\
 prospect view & $x_1 x_2$ view \\
 {\includegraphics[scale=0.22,  clip=true,]{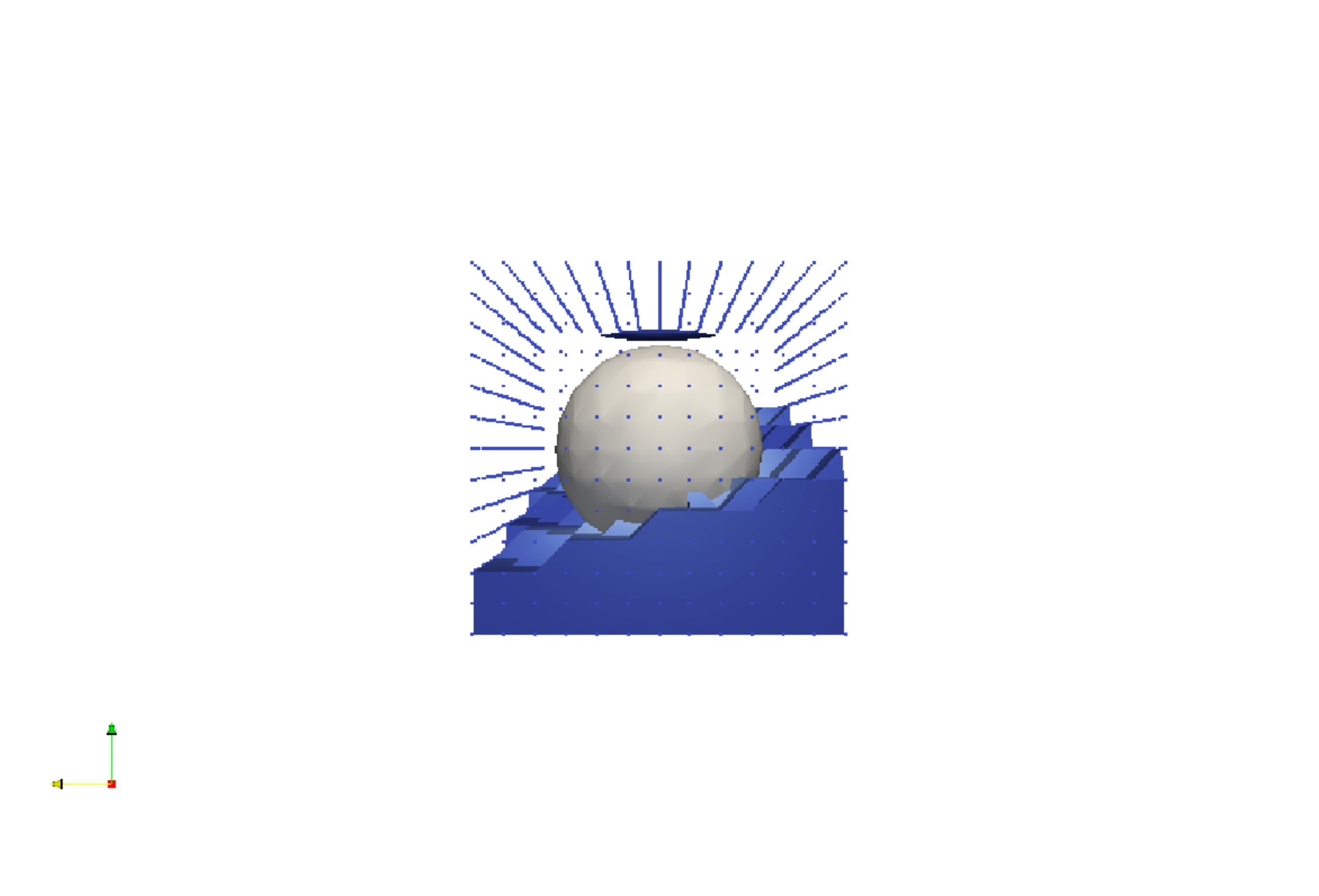}} &
 {\includegraphics[scale=0.22, clip=true,]{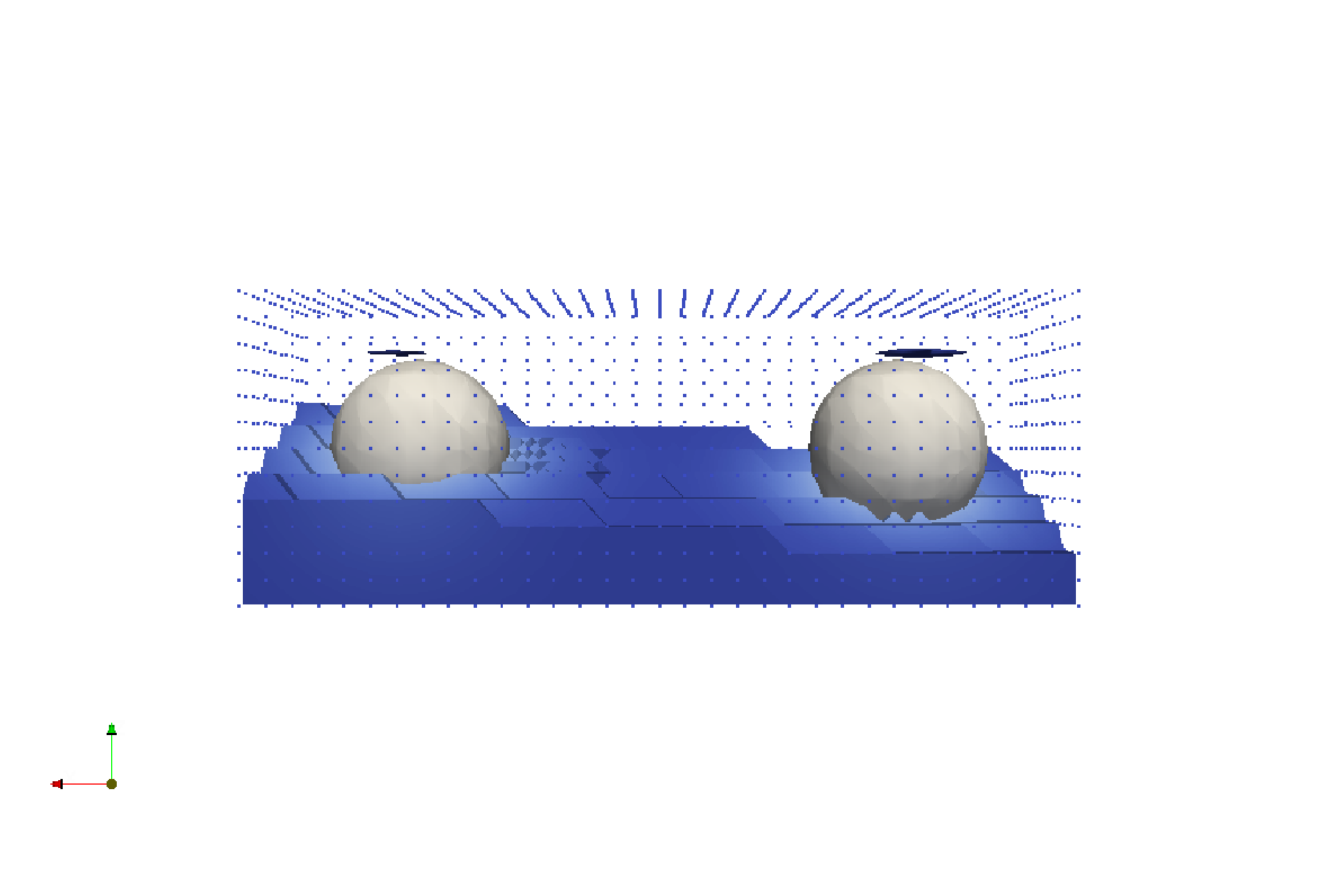}} \\
 $x_2 x_3$ view & $x_3 x_1$ view
 \\
\end{tabular}
 \end{center}
 \caption{ \protect\small \emph{ Test 2. Reconstruction of $c(x)$ with
     $\max_{\Omega_{FEM}} c(x) = 4.1 $ for $\omega=40$ in (\ref{f}) with
     noise level $\sigma=3\%$. We outline also the spherical wireframe
     of the isosurface with exact value of the function
     (\ref{2gaussians}), which corresponds to the value of the
     reconstructed  $c= 0.7\max_{\Omega_{FEM}} c(x)$.} }
 \label{fig:test3noise3a}
 \end{figure}

 \begin{figure}
 \begin{center}
 \begin{tabular}{cc}
 {\includegraphics[scale=0.2,clip=true,]{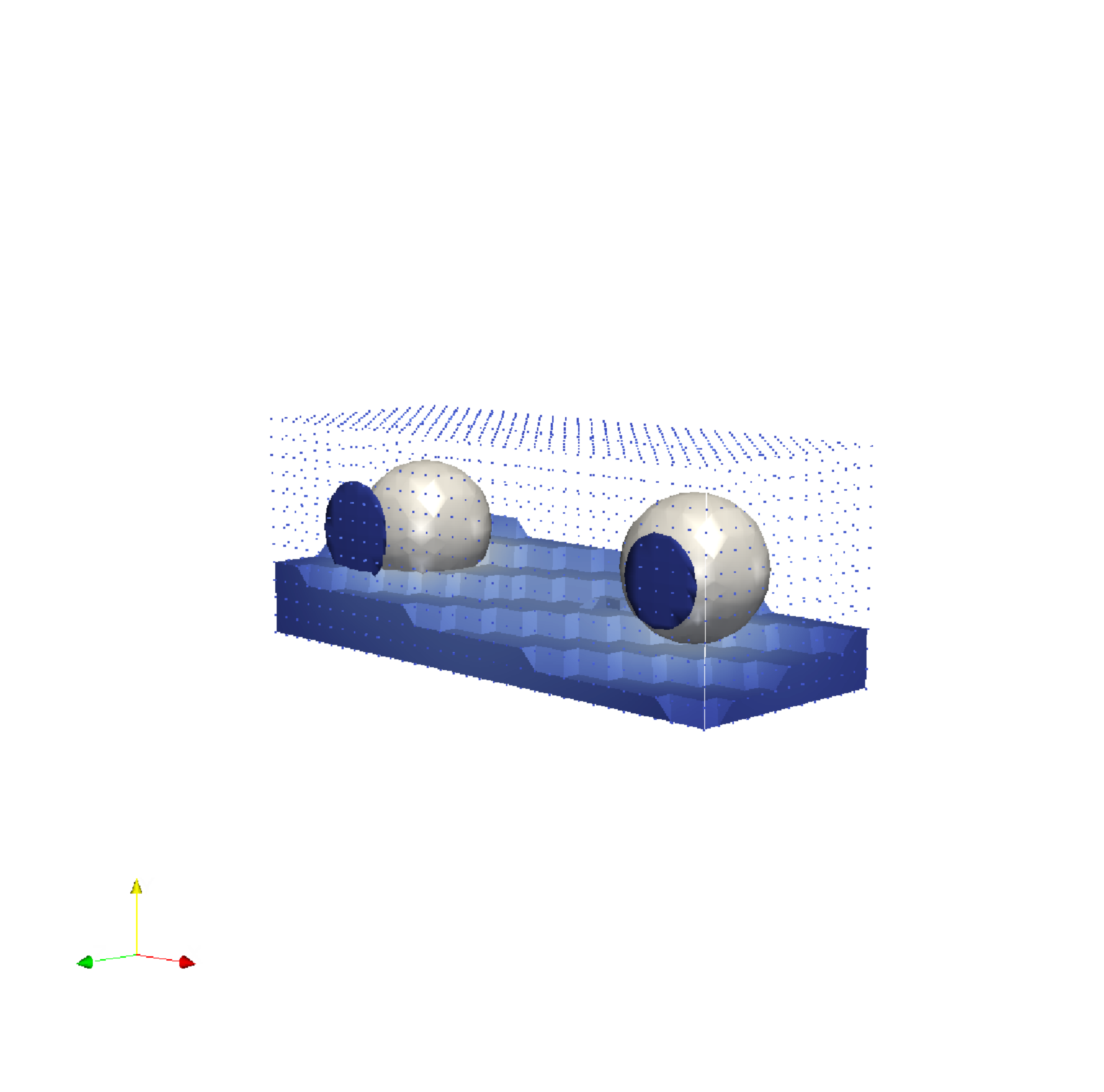}} &
 {\includegraphics[scale=0.15, clip=true,]{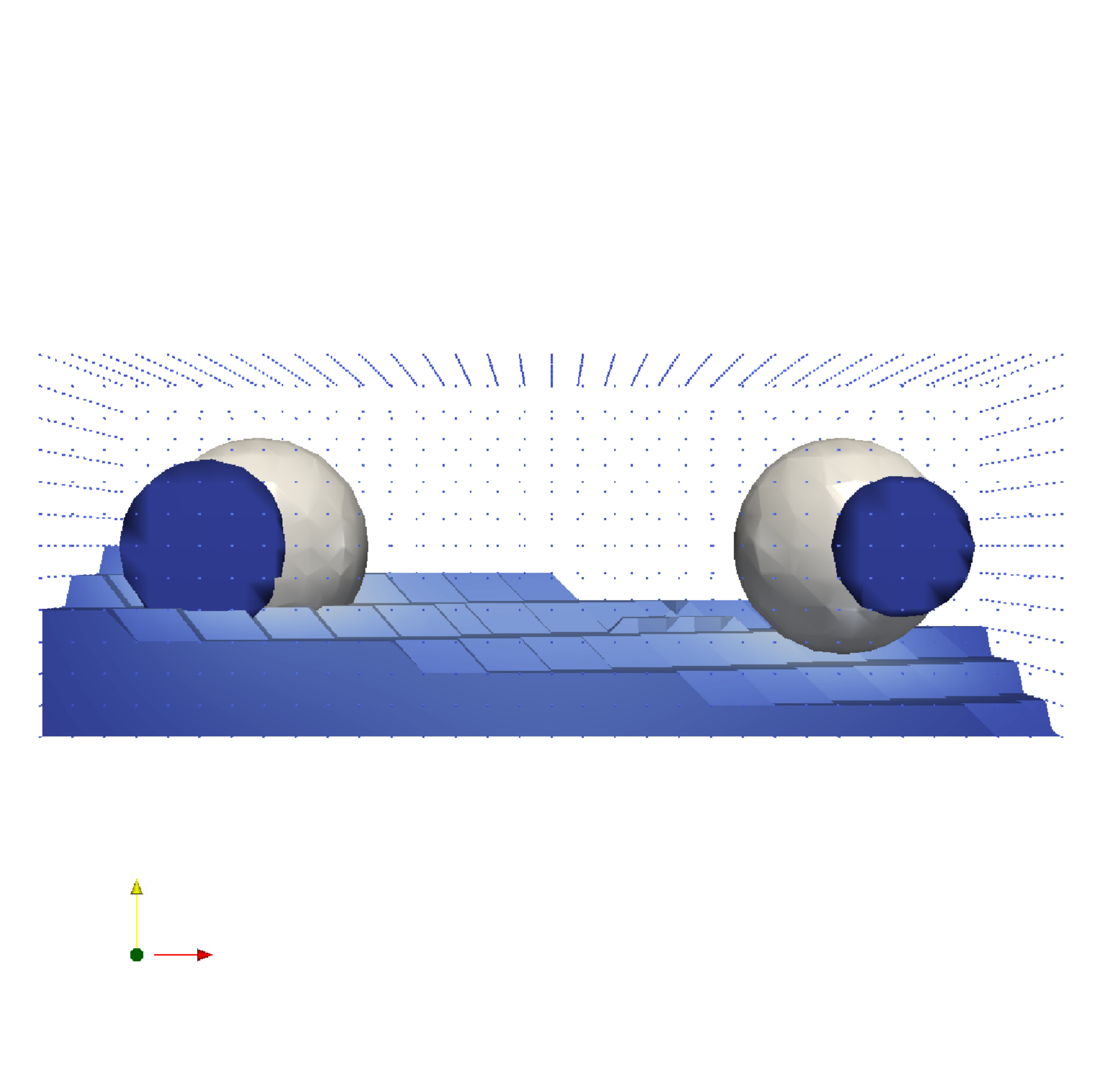}} \\
 prospect view & $x_1 x_2$ view \\
 {\includegraphics[scale=0.15,  clip=true,]{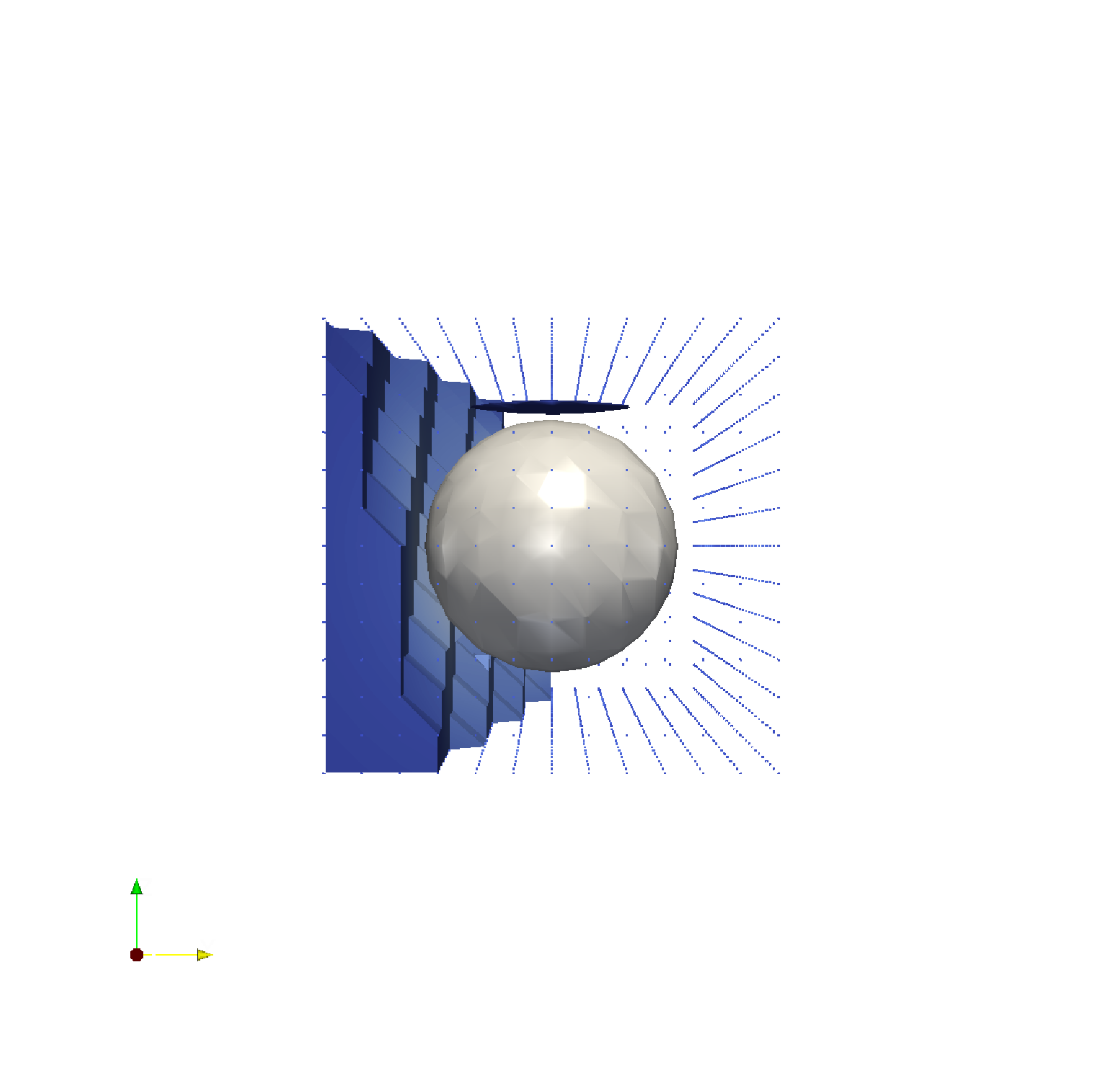}} &
 {\includegraphics[scale=0.15,  clip=true,]{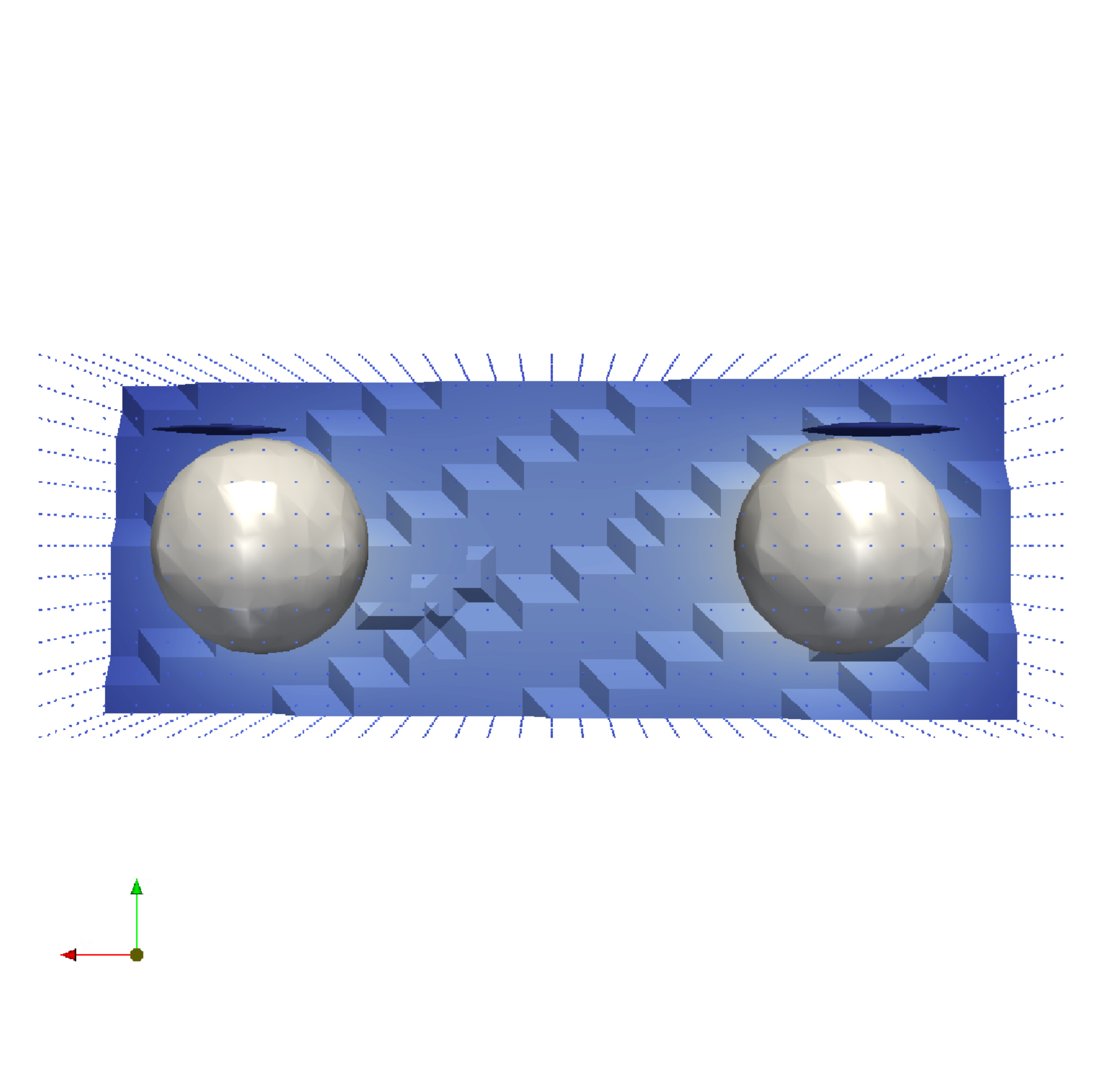}} \\
 $x_2 x_3$ view & $x_3 x_1$ view
 \\
\end{tabular}
 \end{center}
 \caption{ \protect\small \emph{ Test 2. Reconstruction of $c(x)$ with
     $\max_{\Omega_{FEM}} c(x) = 5.58  $ for $\omega=40$ in (\ref{f})
     with noise level $\sigma=10\%$. We outline the spherical wireframe of the
     isosurface with exact value of the function (\ref{2gaussians}),
     which corresponds to the value of the reconstructed  $c= 0.7\max_{\Omega_{FEM}} c(x)$.} }
 \label{fig:test3noise10a}
 \end{figure}

In this numerical test we reconstruct the function $c(x)$ by using
noisy backscattered data.  To get reasonable reconstruction in this
test we runned the conjugade gradient algorithm in time $T=[0,1.5]$
with the time step $\tau=0.003$. We note, that we reduced the
computational time compared with the first test since by running in a
more longer time $T=[0,3.0]$ we have obtained some artifacts at the
middle of the domain. From other side, reducing of the computational
time was resulted in obtaining of a lower contrast in the
reconstructed function, see Table 1.
We tested  reconstruction of the function $c(x)$ with 
 the guess values of $c(x)=1.0$ and
$c_0(x,t)$ as in (\ref{timedepfunc}) since by our assumption this
function is known. 

Figures \ref{fig:test3noise3a}, \ref{fig:test3noise10a} show results
of the reconstruction with $\sigma=3\%$ and $\sigma=10\%$,
respectively.  We observe that the location of the maximal value of
the function (\ref{2gaussians}) is imaged very well.  Again, as in the
previous test, the values of the background in (\ref{2gaussians}) are
smoothed out. Comparing figures with results of reconstruction we
conclude that it is desirable improve  shape of the function  $c(x)$ in $x_3$
  direction.

\section{Conclusions}

\label{sec:concl}

In this work we present theoretical investigations and numerical
studies of the reconstruction of the time and space-dependent
coefficient in an infinite cylindrical hyperbolic domain. In the
theoretical part of this work we derive a local Carleman estimate
which is specially formulated for the infinite cylindrical domain
for the case of time-dependent conductivity function.
  
In the numerical part of the paper we present a computational study of
the reconstruction of function $c(x)$ in a hyperbolic problem
(\ref{model1}) using a hybrid finite element/difference method of
\cite{BAbsorb}.  In our numerical tests, we have obtained stable
reconstruction of the location and contrasts of the function $c(x)$ in
$x_1 x_2$-directions for noise levels $\sigma = 3 \%, 10 \%$ in
backscattered data.  However, size and shape on $x_3$ direction should
still be improved in all test cases.  Similarly with \cite{B,BTKB,BH}
in our future work we plan to apply an adaptive finite element method
in order to get better shapes and sizes of the function $c(x)$ in
$x_3$ direction.

\section*{Acknowledgments}

 The part of the research was done during the sabbatical stay of LB at
 the Institut de Math\'{e}matiques de Marseille, Aix-Marseille
 University, France, which was supported by the sabbatical programme
 at the Faculty of Science, University of Gothenburg, Sweden.


\end{document}